\documentclass[12pt]{amsart}

\usepackage{amsmath}
\usepackage{amsfonts,amssymb}
\usepackage{euscript}
\usepackage{amsthm}
\usepackage[matrix,arrow,curve]{xy}
\usepackage{array}

\numberwithin{equation}{section}
\theoremstyle{plain}
\newtheorem{theorem}{Theorem}[section]
\newtheorem{lemma}[theorem]{Lemma}
\newtheorem{predl}[theorem]{Proposition}
\newtheorem{corollary}[theorem]{Corollary}

\newtheorem{conjecture}[theorem]{Conjecture}

\newtheorem{theorem_intro}{Theorem}

\theoremstyle{definition}
\newtheorem{definition}[theorem]{Definition}
\newtheorem{remark}[theorem]{Remark}
\newtheorem{example}[theorem]{Example}

\newcommand{\R}{\mathbb R}

\newcommand{\Z}{\mathbb Z}

\renewcommand{\P}{\mathbb P}
\newcommand{\D}{\mathcal D}
\newcommand{\EE}{\mathcal E}
\newcommand{\FF}{\mathcal F}
\newcommand{\LL}{\mathcal L}
\newcommand{\TT}{\mathcal T}
\renewcommand{\O}{\mathcal O}
\renewcommand{\k}{\mathsf k}

\newcommand{\xra}{\xrightarrow}
\renewcommand{\le}{\leqslant}
\renewcommand{\ge}{\geqslant}

\newcommand{\bul}{\bullet}

\DeclareMathOperator{\Hom}{\textup{Hom}}

\DeclareMathOperator{\Pic}{\mathrm{Pic}}
\DeclareMathOperator{\End}{\mathrm{End}}
\DeclareMathOperator{\im}{\mathrm{im}}
\DeclareMathOperator{\rank}{\mathrm{rank}}
\DeclareMathOperator{\coh}{\mathrm{coh}}
\DeclareMathOperator{\perm}{\mathrm{tran}}
\DeclareMathOperator{\augm}{\mathrm{augm}}
\DeclareMathOperator{\sh}{\mathrm{sh}}

\begin{document}

\title{On cyclic strong exceptional collections of line bundles on surfaces}

\author{Alexey ELAGIN}
\address{Steklov Mathematical Institute of
Russian Academy of Sciences, Moscow, RUSSIA}
\email{alexelagin@rambler.ru}

\author{Junyan XU}
\address{Department of Mathematics, Indiana University, 831 E. Third St., Bloomington, IN 47405, USA}
\email{xu56@umail.iu.edu}

\author{Shizhuo ZHANG}
\address{Department of Mathematics, Indiana University, 831 E. Third St., Bloomington, IN 47405, USA}
\email{zhang398@umail.iu.edu}


\begin{abstract}We study exceptional collections of line bundles on surfaces. 
We prove that any full cyclic strong exceptional collection  of line bundles on a rational surface is an augmentation in the sense of L.\,Hille and M.\,Perling. 
We find simple geometric criteria of exceptionality (strong exceptionality, cyclic strong exceptionality) for collections of line bundles on  weak del Pezzo surfaces.
As a result, we classify smooth projective surfaces admitting a full cyclic strong exceptional collection  of line bundles.  Also, we provide an example of a weak del Pezzo surface of degree $2$ and a full strong exceptional collection of line bundles on it which does not come from augmentations, thus answering a question by Hille and Perling. 
\end{abstract}

\maketitle


\tableofcontents

\section{Introduction}

Our paper is devoted to the study of exceptional collections of line bundles on surfaces. 
By a surface we always mean a smooth connected projective surface over an algebraically closed field of zero characteristic.
Among the questions that are addressed in this paper are the following

\medskip
\textbf{Question 1.} Which surfaces admit exceptional/strong exceptional/cyclic strong exceptional collections of line bundles?

\textbf{Question 2.}  How to construct exceptional/strong exceptional/cyclic strong exceptional collections of line bundles if they exist?

\textbf{Question 3.} How to tell whether a given collection of line bundles on a surface is exceptional/strong exceptional/cyclic strong exceptional?

\medskip
It is believed that any variety with a full exceptional collection in the bounded derived category of coherent sheaves is rational. Though, there is no proof yet even in the case when the collection is formed by line bundles. On the other hand, on any rational surface one can construct a full exceptional collection of line bundles, using a construction by Dmitry Orlov \cite{Or}. 

For strong exceptional collections of line bundles the question is much more complicated.
First, it is known that any surface with a full strong exceptional collection of line bundles is rational, see \cite{BS}.
It was conjectured by Alastair King \cite{Ki} that any smooth toric variety has a strong exceptional collection formed by line bundles. In \cite{HP2} Lutz Hille and Markus Perling constructed a surface which is a counterexample.
Later in \cite[Theorem 8.2]{HP} they established a criterion determining whether a toric surface  admits a full strong exceptional collection of line bundles or not. It was demonstrated  that such a collection exists if and only if the toric surface can be obtained from some Hirzebruch surface by two blow-up operations: on each step one can blow up several distinct points. Also it was 
explained that on any (not necessarily toric) rational surface obtained from a Hirzebruch surface by two blow-up operations there exists a full strong exceptional collection of line bundles, see Theorem 5.9 in \cite{HP}. Therefore one arrives at a reasonable

\begin{conjecture}[\cite{HP}]
\label{conj_twosteps}
A rational surface $X$ has a full and strong exceptional collection of line bundles in the derived category if and only if $X$ can be obtained from some Hirzebruch surface by at most two steps of blowing up points (maybe several at each step).
\end{conjecture}

\medskip
Of special interest are cyclic strong exceptional collections of line bundles. Recall that a strong  exceptional collection $(\O_X(D_1),\ldots,\O_X(D_n))$ is \emph{cyclic strong} if it remains strong exceptional after ``cyclic shifts''. That is, any segment 
$(\O_X(D_{k+1}),\ldots,\O_X(D_{k+n}))$ in the infinite helix $\ldots,\O_X(D_i),\ldots$, defined by the rule $D_{i+n}=D_i-K_X$, is also strong exceptional. 
Cyclic strong exceptional collections of line bundles and the tilting bundles formed by them have attracted much attention by specialists from different areas of mathematics  and have been  studied under different names (pull-back exceptional collections in \cite{Br}, \cite{BF}, \cite{VdB}, first order approximation to pull-back exceptional collections in  \cite{BF}, $2$-hereditary exceptional collections in \cite{Ch}). For example, 
let $X$ be a smooth del Pezzo surface and $Y$ be the
total space of the canonical line bundle on $X$. 
Then a full cyclic strong
exceptional collection of line bundles on $X$  will pull back to a tilting
bundle $T$ on $Y$. The algebra $\End(T)$ gives a noncommutative crepant resolution of the anti-canonical cone of $X$ in the sense of Michel Van den Berg, see \cite{VdB}.

In Hille and Perling's paper \cite{HP} the following results concerning full cyclic strong exceptional collections of line bundles were obtained:
\begin{itemize}
\item if a rational surface $X$ has a full cyclic strong exceptional collection of line bundles then \\$\deg(X):=K_X^2\ge 3$;
\item any del Pezzo surface $X$ with $\deg(X)\ge 3$ has a full cyclic strong exceptional collection of line bundles;
\item a toric surface $X$ has a full cyclic strong exceptional collection of line bundles if and only if $-K_X$ is numerically effective.
\end{itemize}

In this paper we generalize the above results and give a complete classification of surfaces admitting a full cyclic strong exceptional collection of line bundles, answering Question 1 for such collections. First, any such surface is a weak del Pezzo surface.
Recall that a weak del Pezzo surface is a smooth projective surface $X$ such that $K_X^2>0$ and $-K_X$ is numerically effective. Alternatively, $X$ is the minimal resolution of singularities on a singular del Pezzo surface having only rational double points as singularities. Weak del Pezzo surfaces are rational, they are distinguished by their \emph{type} (see Section~\ref{section_wdp}).
Second, for any type of weak del Pezzo surfaces we determine whether they possess a full cyclic strong exceptional collection of line bundles.

\begin{theorem_intro}[Propositions~\ref{prop_cyclicwdp} and \ref{prop_classification}]
\label{theorem_intro1}
Let $X$ be a smooth projective surface having a cyclic strong exceptional collection of line bundles of maximal length. Then $X$ is a weak del Pezzo surface and $X$ is one of the surfaces from Table~\ref{table_yes}. Moreover, any weak del Pezzo surface from  Table~\ref{table_yes} possesses a full cyclic strong exceptional collection of line bundles.
\end{theorem_intro}

Basing on Theorem~\ref{theorem_intro1}, in a recent paper \cite{Zh} 
some new examples of non-commutative crepant resolutions in the sense of \cite{VdB} were constructed.

For the study of exceptional collections of line bundles, we use  two notions, proposed by Hille and Perling.
First, it is convenient to pass from an exceptional collection 
$$(\O_X(D_1),\ldots,\O_X(D_n))$$ 
of line bundles (where $D_i$ are divisors on $X$) to the sequence
$A_1,\ldots,A_{n-1}$ of differences: $A_k=D_{k+1}-D_k$ for $k=1,\ldots,n-1$. It is useful also 
to complete the sequence $D_1,\ldots,D_n$ to an infinite helix $D_i, i\in \Z$, by the rule $D_{k+n}=D_k-K_X$ for all $k$, and
to add the term
$A_n=D_{n+1}-D_n=D_1-K_X-D_n$ to the sequence $A_1,\ldots,A_{n-1}$. Thus one gets a sequence
$$A=(A_1,\ldots,A_n), \quad A_i\in\Pic (X),$$
which contains essentially the same information as the original collection 
$(\O_X(D_1),\ldots,\O_X(D_n))$. 
Elements $A_1,\ldots,A_n$ generate $\Pic(X)$ and have nice combinatorial properties (where indices are treated modulo $n$):
\begin{itemize}
\item $A_i\cdot A_{i+1}=1$;
\item $A_i\cdot A_j=0$ if $j\ne i,i\pm 1$;
\item $\sum_i A_i=-K_X$.
\end{itemize}
These properties comprise the definition of a \emph{toric system} $A_1,\ldots,A_n$, see Definition~\ref{def_ts}.

To answer Question 2, we use the second discovery of \cite{HP}, the notion of augmentation. This operation enables, starting from a toric system $A=(A_1,\ldots,A_n)$ on some surface $X$, to obtain a toric system on the blow-up $X'$ of $X$ at a point. Namely, let $1\le m\le n+1$ be an index and let $E$ denote the exceptional divisor of the blow-up. Then the augmented toric system  is the sequence
$$\augm_m(A)=(A'_1,\ldots,A'_{m-2},A'_{m-1}-E,E,A'_m-E,A'_{m+1},\ldots,A'_n)$$
in $\Pic (X')$,
where $A'_i$ is the pull-back of $A_i$. One can start from some toric system on a Hirzebruch surface and perform several blow-ups, augmenting the toric system at each step. The resulting toric system is called a \emph{standard augmentation}. This construction gives us some (quite many) explicit examples of toric systems on rational surfaces. 

It is natural to ask a question: does any full exceptional collection of line bundles correspond to a toric system which is a standard augmentation? The above cannot be literally true, but is close to being true if we do not distinguish between exceptional collections which differ only by the ordering of line bundles inside blocks of completely orthogonal bundles. Reordering of two mutually orthogonal line bundles in the exceptional collection corresponds to an operation with toric systems which we call \emph{transposition}. By a \emph{permutation} of a toric system we mean a composition of several transpositions, see Section~\ref{section_prelim3} for details.

A positive answer to the above question means that we are able to find all full exceptional collections of line bundles: any collection on a rational surface $X$ is a standard augmentation for some minimal model $X\to \mathbb F_d$ and some exceptional collection on $\mathbb F_d$. For any fixed minimal model and exceptional collection on $\mathbb F_d$ there exists only a finite number of such standard augmentations.

Besides that, a positive answer to the above question helps to solve Conjecture~\ref{conj_twosteps}.
For standard augmentations Hille and Perling proved the following: a toric system on a rational surface $X$ which is a standard augmentation along some sequence of blow-ups $X\to\ldots\to \mathbb F_d$ corresponds to a strong exceptional collection only if $X$ was obtained from the Hirzebruch surface $\mathbb F_d$ in at most two blow-ups (each time one can blow up several different points).   

Motivated by that result,
Hille and Perling made the following 
\begin{conjecture}
\label{conj_augm}
Any full strong exceptional collection of line bundles on a rational surface corresponds  (up to permutation of completely orthogonal bundles) to a standard augmentation. 
\end{conjecture} 
This conjecture can be viewed as a ``homological version'' of the minimal model program for rational surfaces.

Note that Conjecture~\ref{conj_augm} implies Conjecture~\ref{conj_twosteps}.

Hille and Perling proved Conjecture~\ref{conj_augm} for toric surfaces, thus proving for toric surfaces Conjecture~\ref{conj_twosteps}. Andreas Hochenegger and Nathan Ilten in \cite[Main Theorem 3]{HI} proved Conjecture~\ref{conj_augm} for toric surfaces of Picard rank $\le 4$ without strongness assumption. The first author and Valery Lunts in \cite[Theorem 1.4]{EL} proved  Conjecture~\ref{conj_augm} for del Pezzo surfaces and arbitrary numerically exceptional collections of line bundles of maximal length. 

In this paper we prove Conjecture~\ref{conj_augm} for cyclic strong exceptional collections. 

\begin{theorem_intro}[See Theorem~\ref{theorem_1kindweak} and Corollary~\ref{cor_cyclicaugm} for the precise statements]
\label{theorem_intro2}
Let $X$ be a smooth rational projective surface and $A=(A_1,\ldots,A_n)$ be a toric system on $X$. Suppose that $A$ is numerically cyclic strong: that is, the Euler characteristic $\chi(X,A_i)\ge 0$ for all $i$. Then up to some permutations $A$ is a standard augmentation.
In particular, any full cyclic strong exceptional collection of line bundles on $X$ corresponds (up to some permutations) to a toric system which is a standard augmentation.
\end{theorem_intro}

Therefore, one is able to list all full cyclic strong exceptional collections on a given surface.  This gives an answer to Question 2 for full cyclic strong exceptional collections.

\medskip
Our initial idea was to  investigate Conjecture~\ref{conj_augm} for weak del Pezzo surfaces. 
We have proved Conjecture~\ref{conj_augm} for all weak del Pezzo surfaces of degree $\ge 3$. 
\begin{theorem_intro}
\label{theorem_intro3}
Let $X$ be a weak del Pezzo surface of degree $\ge 3$. Then  any  full strong exceptional collection of line bundles on $X$ corresponds (up to some permutations) to a standard augmentation.
\end{theorem_intro}

The proof is heavily technical and uses some machine computations, we do not give it here. It can be found in the extended version of this paper, see \cite{EXZ}.

What is more interesting, we found that in general Conjecture~\ref{conj_augm} is \textbf{false}. That is, we have found counterexamples to Conjecture~\ref{conj_augm} for some weak del Pezzo surfaces of degree $2$. 
\begin{theorem_intro}[See Section~\ref{section_ce} for details]
\label{theorem_4}
Let $X$ be a weak del Pezzo surface of degree $2$ of type $A_1+2A_3$. There exists a full strong exceptional collection of line bundles on $X$ such that the corresponding toric system  is not a standard augmentation (up to any permutations).
\end{theorem_intro}

The example from Theorem~\ref{theorem_4} was found using a computer, as well as many other counterexamples to Conjecture~\ref{conj_augm}.
It should be noted that any of the surfaces where a counterexample was found can be obtained from a Hirzebruch surface by two blow-ups.
This gives some evidence supporting Conjecture~\ref{conj_twosteps}. It is interesting also that all surfaces  where we found a counterexample to Conjecture~\ref{conj_augm} have holes in the effective cone: such non-effective divisor classes which have a positive multiple being effective.

There are some examples known in the literature of strong exceptional  collections of line bundles which are not standard augmentations, see \cite{HI} and \cite{Ho}. But some permutations in the cited  collections are standard augmentations, thus the cited examples agree with Conjecture \ref{conj_augm}. See Remark~\ref{remark_HI} for more details.

\medskip

In order to construct such a counterexample or to construct examples of cyclic strong exceptional toric systems (see Theorem~\ref{theorem_intro1}) one needs to have a reasonable answer to Question 3: how to check that a given collection of line bundles on a surface is exceptional/strong exceptional? A priori one needs to check cohomology vanishing for lots of line bundles. To do this effectively, we propose a simple geometric criterion for exceptionality of toric systems, see Theorem~\ref{theorem_checkonlyminustwo}. For cyclic strong exceptional toric systems  it says:
\begin{theorem_intro}
Let $X$ be a weak del Pezzo surface. Let $A=(A_1,\ldots,A_n)$ be a toric system on $X$. Then $A$ corresponds to a cyclic strong exceptional collection of line bundles if and only if the following holds:
\begin{itemize}
\item $A_i^2\ge -2$ for all $i=1,\ldots,n$;
\item for any cyclic segment $[k,\ldots,l]\subset [1,\ldots, n]$ (see Definition~\ref{def_kl}) such that $A_k^2=A_{k+1}^2=\ldots=A_{l-1}^2=A_l^2=-2$  the divisor $A_k+\ldots+A_l$ is neither effective nor anti-effective.
\end{itemize}
\end{theorem_intro}

Note that the divisors $D=\pm(A_k+\ldots +A_l)$ in the above theorem are $(-2)$-classes: they satisfy equalities $D^2=-2, D\cdot K_X=0$. It is a simple observation that such $D$ is effective if and only if it is a sum of some $(-2)$-curves on $X$. For any type of weak del Pezzo surfaces  the configuration
of $(-2)$-curves on it is known. Thus it is quite simple to tell
whether a toric system on a weak del Pezzo surface is cyclic strong exceptional or not.

\bigskip
Let us mention some applications of our results.

Matthew Ballard and David Favero in \cite{BF} discovered a relation between cyclic strong exceptional collections in the derived category $\D^b(\coh X)$ of coherent sheaves on a smooth projective variety $X$ and dimension of the category $\D^b(\coh X)$ in the sense of Rouquier. They demonstrated, in particular, that a full strong exceptional collection $(\EE_1,\ldots,\EE_n)$ in $\D^b(\coh X)$ is cyclic strong if the generator $\oplus_i\EE_i$ has generation time equal to $\dim X$. 
For collections of line bundles on surfaces we are able to provide a sort of  converse statement.

\begin{theorem_intro}[Proposition \ref{prop_timetwo}]
\label{theorem_intro6}
Let $X$ be a smooth projective surface with a full cyclic strong exceptional collection
$(\LL_1,\ldots,\LL_n)$
of line bundles. Then the generator $\oplus_i\LL_i$ of the category $\D^b(\coh X)$ has the generation time two. In particular, for any surface $X$ from Table~\ref{table_yes} 
the category $\D^b(\coh X)$ has dimension two.
\end{theorem_intro}

The above theorem provides a new class of varieties $X$ for which the dimension of $\D^b(\coh X)$ equals to the dimension of $X$. This supports a conjecture by Orlov saying that $\dim \D^b(\coh X)=\dim X$ for all smooth projective varieties $X$.

\medskip
Let us briefly mention one other application of our results. Basing on Theorems \ref{theorem_intro1} and \ref{theorem_intro2}, it is demonstrated by the third author in \cite[Theorem 1.9]{Zh} that a full exceptional  collection of line bundles on a smooth projective surface is cyclic strong exceptional iff it is a pull-back exceptional collection (\cite{Br}, \cite{BF}, \cite{VdB}), iff it is a first order approximation to a pull-back exceptional collection in the sense of \cite{BF} and iff it is a $2$-hereditary exceptional collection (\cite{Ch}).

\bigskip
The text is organized as follows. Sections \ref{section_prelim1},\ref{section_prelim2}, and \ref{section_prelim3} contain preliminary material. Here we recall, introduce and discuss necessary notions: divisors on surfaces, their orthogonality properties, $r$-classes, weak del Pezzo surfaces, exceptional collections and toric systems,  admissible sequences, augmentations, transpositions. This material is partially known and/or contained in the paper \cite{HP} by Hille an Perling. Let us point out what is new.

First, it is Proposition~\ref{prop_loslo} providing a handy geometric tool for checking whether a given divisor on a weak del Pezzo surface is left-orthogonal or strong left-orthogonal. Further, a new results is Theorem~\ref{theorem_checkonlyminustwo}. This theorem, in terms of the effective cone, gives an instrument for checking ((cyclic) strong) exceptionality of a given toric system on a weak del Pezzo surface.

Also, in Section \ref{subsection_augm} we define and discuss different variants of augmentation.
We distinguish between elementary augmentations, standard augmentations, augmentations in the weak sense, exceptional augmentations,  strong exceptional augmentations, and  cyclic strong exceptional augmentations. 
Standard augmentations are the most natural but they do not exhaust all toric systems on most surfaces if we do not allow permutations. Augmentations in the weak sense
seem to be a good notion if we do not care about homological properties of collections like exceptionality. We argue that strong exceptional augmentations are suitable for accurate formulation of Conjecture~\ref{conj_augm}: any strong exceptional toric system is a strong exceptional augmentation. 

Note here that terminology in \cite{HP} is different: they say that a strong exceptional collection  has a \emph{normal form} that is a standard augmentation. Normal form of a collection is obtained from the original collection by permutations. Unfortunately, the corresponding definitions in \cite{HP} are not rigorous enough, this forced us to introduce the certain notions of augmentations.

In Section \ref{section_firstkind} we treat toric systems of the first kind: such toric systems $(A_1,\ldots,A_n)$ that $\chi(A_i)\ge 0$ for all $i$. For example, toric systems corresponding to cyclic strong exceptional collections are of the first kind. 
Essentially, we prove that toric systems of the first kind are augmentations in  certain sense. In particular, we prove Theorem~\ref{theorem_intro2}.

Section \ref{section_cyclicstrong} is devoted to the proof of Theorem~\ref{theorem_intro1}. 

In Section~\ref{section_application} we give an application of our results to the study of dimension of derived categories of coherent sheaves. In this section we prove Theorem~\ref{theorem_intro6}.
 
In Section \ref{section_ce} we present a counterexample to Conjecture \ref{conj_augm}. 

Finally, in Appendix we present classification of weak del Pezzo surfaces of degree $\ge 3$. It is needed for Section~\ref{section_cyclicstrong}.

\medskip
{\bf Acknowledgements.} The authors would like to thank Valery Lunts, collaboration with whom initiated this project. We thank Michael Larsen for valuable remarks on the text. The first author is grateful to Indiana University for its hospitality. The third author would like to thank Li Tang for his support. Finally, we are extremely grateful to the referees for the careful reading of the manuscript and for many useful remarks and suggestions.

\section{Divisors on surfaces and their properties}
\label{section_prelim1}

Throughout the paper we assume that the base field $\k$ is algebraically closed and has zero characteristic. All surfaces we consider are supposed to be smooth projective and connected.

\subsection{Divisors and $r$-classes}
\label{subsection_rclasses}

Let $X$ be  a smooth projective surface over $\k$.  Let $K_X$ be a canonical divisor on $X$.
Let $d=K_X^2$ be the degree of $X$, further we always assume that $d>0$. 
For a divisor $D$ on $X$, we will use the following shorthand notations:
$$H^i(D):=H^i(X,\O_X(D)),\quad h^i(D)=\dim H^i(D),\quad \chi(D)=h^0(D)-h^1(D)+h^2(D).$$
By the Riemann-Roch formula, one has
$$\chi(D)=\chi(\O_X)+\frac{D\cdot (D-K_X)}2.$$

The following notions are introduced by Lutz Hille and Markus Perling in \cite[Definition 3.1]{HP}. 
\begin{definition}
\begin{itemize}
\item
A divisor $D$ on $X$ is \emph{numerically left-orthogonal} if $\chi(-D)=0$.
\item
A divisor $D$ on $X$ is \emph{left-orthogonal}  if $h^i(-D)=0$ for all $i$. 
\item
A divisor $D$ on $X$ is \emph{strong left-orthogonal} if $h^i(-D)=0$ for all $i$ and $h^i(D)=0$ for $i\ne 0$.
\end{itemize}
\end{definition} 

\begin{definition}
We call $D$ an \emph{$r$-class} if $D$ is numerically left-orthogonal and $D^2=r$. \end{definition}

Motivation: if $C\subset X$ is a smooth rational irreducible curve, then the class of $C$ in $\Pic(X)$ is an $r$-class where $r=C^2$. 

If $C$ is a smooth rational irreducible curve on $X$ and $r=C^2$, it is said that $C$ is an \emph{$r$-curve}. An $r$-curve is \emph{negative} if $r<0$.

The next propositions are easy consequences of Riemann-Roch formula, see~\cite[Lemma 3.3]{HP} or \cite[Lemma 2.9, Lemma 2.10]{EL}.
\begin{predl}
\label{prop_D2}
Let $X$ be a surface  with $\chi(\O_X)=1$.
Then a divisor $D$ on $X$ is numerically left-orthogonal if and only if
$$D^2+2=-D\cdot K_X.$$
In particular, $D$ is an $r$-class if and only if
$$D^2=r, \quad D\cdot K_X=-r-2.$$
For any numerically left-orthogonal divisor $D$ on $X$ one has
$$\chi(D)=D^2+2=-D\cdot K_X.$$
\end{predl}

\begin{predl}
\label{prop_classes0}
Let $X$ be a surface  with $\chi(\O_X)=1$.
Suppose $D_1,D_2$ are numerically left-orthogonal divisors on $X$. 
Then $D_1+D_2$ is numerically left-orthogonal if and only if $D_1D_2=1$. If that is the case, then 
$$\chi(D_1+D_2)=\chi(D_1)+\chi(D_2)\qquad \text{and}\qquad (D_1+D_2)^2=D_1^2+D_2^2+2.$$
\end{predl}

\begin{lemma}
\label{lemma_dprime}
Let $D$ be an $r$-class on a surface $X$ of degree $d$ with $\chi(\O_X)=1$. Then $D'=-K_X-D$ is an $r'$-class where $r+r'=d-4$. 
\end{lemma}
\begin{proof}
Indeed, 
$$(-K_X-D)^2+(-K_X-D)\cdot K_X=D^2+D\cdot K_X=-2$$
and thus $D'$ is numerically left-orthogonal by Proposition~\ref{prop_D2}.
Also, using Proposition~\ref{prop_D2} we get $(D')^2=K_X^2+2D\cdot K_X+D^2=d+(-2r-4)+r=d-4-r$.
\end{proof}

Denote the set of $(-1)$-classes on $X$ by $I(X)$ and the set of $(-2)$-classes on $X$ by $R(X)$. As we will see, $I(X)$ and $R(X)$ depend only on $\deg X$ and do not depend on geometry of $X$.

Suppose  that $X$ is a blow-up of $\P^2$ at $n$ (maybe infinitesimal) points. Then the Picard group $\Pic(X)$ of $X$ is a finitely generated abelian group with the standard  basis $L,E_1,\ldots,E_n$. 
We will often use the following shorthand notation:
\begin{equation}
\label{eq_Lij}
E_{i_1\ldots i_k}:=E_{i_1}+\ldots+E_{i_k},\quad L_{i_1\ldots i_k}:=L-E_{i_1\ldots i_k}.
\end{equation}
The intersection form in the basis  $L,E_1,\ldots,E_n$ is given by $L^2=1, E_i^2=-1, L\cdot E_i=0, E_i\cdot E_j=0$ for $i\ne j$. We have $K_X=-3L+E_1+\ldots+E_n$ and $\deg(X)=9-n$. That is, the abelian group $\Pic(X)$ with the intersection form and the distinguished element $K_X$ depend only on $n$ and do not depend on the geometry of $X$. Since $I(X)$ and $R(X)$ are determined by numerical conditions $D\cdot K_X=D^2=-1$ and $D\cdot K_X=0, D^2=-2$ respectively, they also depend only on $\deg(X)$ and do not depend on geometry of $X$.
If  $\deg(X) \ge 1$ (or $n\le 8$), then $R(X)$ is a root system in some subspace in $N_X=(K_X)^{\perp}\subset \Pic(X)\otimes\R$ (by  \cite[Theorem 23.9]{Ma} for $1\le \deg(X)\le 6$, trivial for $\deg(X)=7,8,9$). Moreover, if $\deg(X)\le 6$ then $R(X)$ spans $N_X$. 

\begin{table}[h]
\caption{Root systems $R(X)$ for blow-ups $X$ of $\P^2$ for  $1\le \deg(X)\le 9$}
\begin{center}
\begin{tabular}{|c|c|c|c|c|c|c|c|c|c|}
		 \hline
			degree of $X$ & 9& 8& 7 & 6& 5&4&3&2&1\\
		 \hline
			type & $\emptyset$ & $\emptyset$ & $A_1$ & $A_1+A_2$ & $A_4$ & $D_5$ & $E_6$ & $E_7$ & $E_8$\\
		 \hline	
		  $|R(X)|$ & $0$ & $0$ & $2$ & $8$ & $20$ & $40$ & $72$ & $126$ & $240$\\
		 \hline	
			$|I(X)|$ & $0$ & $1$ & $3$ & $6$ & $10$ & $16$ & $27$ & $56$ & $240$\\
		 \hline	
\end{tabular}
\end{center}

\label{table_root}
\end{table}

\begin{definition}
\label{def_effirr}
Denote by $R^{\rm eff}(X)\subset R(X)$ the subset of effective $(-2)$-classes. Denote by $R^{\rm irr}(X)\subset R^{\rm eff}(X)$ and $I^{\rm irr}(X)\subset I(X)$ the subsets of classes of $(-2)$-curves and $(-1)$-curves respectively. 
Denote by $R^{\rm slo}(X)\subset R^{\rm lo}(X)\subset R(X)$ the subsets of strong left-orthogonal and left-orthogonal $(-2)$-classes respectively. 
\end{definition}
Note that  $R^{\rm eff}(X),R^{\rm irr}(X),I^{\rm irr}(X), R^{\rm lo}(X)$ and $R^{\rm slo}(X)$ essentially depend on the surface $X$, in contrast with $R(X)$.

\subsection{Weak del Pezzo surfaces}
\label{section_wdp}
By definition, a weak del Pezzo surface is a smooth connected projective rational surface~$X$ such that $K_X^2>0$ and $-K_X$ is nef. A del Pezzo surface is a  smooth connected projective rational surface $X$ such that $-K_X$ is ample.  We refer to Igor Dolgachev~\cite[Chapter 8]{Do} or Ulrich Derenthal~\cite{De} for the main properties of weak del Pezzo surfaces. A weak del Pezzo surface is a del Pezzo surface if and only if it has no $(-2)$-curves. Every weak del Pezzo surface $X$ except for Hirzebruch surfaces $\mathbb F_0$ and $\mathbb F_2$ is a blow-up of $\P^2$ at several (maybe infinitesimal) points. That is, there exists a sequence 
$$X=X_n\xra{p_n} X_{n-1}\to \ldots\to  X_1\xra{p_1} X_0=\P^2$$
of $n$ blow-ups, where $p_k$ is the blow-up of point $P_k\in X_{k-1}$. 
Moreover, the surface $X_n$ as above is a weak del Pezzo surface if and only if $n\le 8$ and for any $k$ the point $P_k$ does not belong to a $(-2)$-curve on $X_{k-1}$.

\begin{lemma}
\label{lemma_negcurve}
Let $C\subset X$ be a reduced irreducible curve with $C^2<0$ on a weak del Pezzo surface $X$.  Then $C$ is a smooth rational curve and $C^2=-2$ or $-1$. If $X$ is a del Pezzo surface then $C^2=-1$.
\end{lemma} 
\begin{proof}
By the Riemann-Roch formula, one has 
$$\chi(\O_C)=\chi(\O_X)-\chi(\O_X(-C))=-\frac{C^2+C\cdot K_X}{2}.$$
Since $C$ is irreducible and reduced, one has $h^0(\O_C)=1$. Hence
$$0\le h^1(\O_C)=1+\frac{C^2+C\cdot K_X}2.$$
Note that $C\cdot K_X\le 0$ because $-K_X$ is nef and $C^2<0$ by assumptions. It follows that $h^1(\O_C)=0$ and $p_a(\O_C)=0$. Hence $C$ is rational and smooth by \cite[Exer. IV.1.8]{Ha}. Moreover, $C^2+C\cdot K_X=-2$, therefore $C^2=-2$ or $-1$. If $X$ is a del Pezzo surface, then $C\cdot K_X<0$ and the case $C^2=-2$ is impossible.
\end{proof}

For the future use we formulate
\begin{predl}[See {\cite[Section 8.2.7]{Do}}]
\label{prop_roots}
Let $X$ be a weak del Pezzo surface. Then there exists a root subsystem $\Phi$ in the root system $R(X)$ such that the set $R^{\rm irr}(X)$ of classes of $(-2)$-curves is the set of simple roots in $\Phi$ and  the set $R^{\rm eff}(X)$ of effective $(-2)$-classes  is the set of positive roots in $\Phi$. Moreover, $\Phi$ is a disjoint union of root systems of types $A,D,E$.
\end{predl}

Two weak del Pezzo surfaces $X$ and $Y$ are said to have the same \emph{type} if there exists an isomorphism $\Pic(X)\to \Pic(Y)$ preserving the intersection form, the canonical class and identifying the sets of negative curves. Thus, for a  weak del Pezzo surface it makes sense to consider the configuration (i.e., the incidence graph) of irreducible $(-2)$-curves on it. By Proposition~\ref{prop_roots}, this configuration is always a disjoint union of Dynkin graphs of types $A,D,E$ and will be denoted respectively.
In most cases, the configuration of irreducible $(-2)$-curves determines the type uniquely. Surfaces of different types and with the same configuration of irreducible $(-2)$-curves are distinguished by the number of irreducible $(-1)$-curves. 
A surface of degree $d$ with the configuration of $(-2)$-curves $\Gamma$ and with $m$ irreducible $(-1)$-curves is denoted by $X_{d,\Gamma,m}$. Number $m$ is omitted if $d$ and $\Gamma$ determine the type uniquely.
For example, $X_{4,A_2}$ denotes a weak del Pezzo surface of degree $4$ with $(-2)$-curves forming the Dynkin graph $A_2$ and $X_{4,A_3,5}$ denotes a weak del Pezzo surface of degree $4$ with $(-2)$-curves forming the Dynkin graph $A_3$ and with five $(-1)$-curves. 

Therefore, all weak del Pezzo surfaces of the same type have ``the same'' sets of negative curves. In Appendix A we list all possible types of weak del Pezzo surfaces of degrees from $3$ to $7$, for any type we specify the sets of $(-2)$-curves.

\medskip

Now we formulate	criteria for checking left-orthogonality and strong left-orthogonality of divisors on weak del Pezzo surfaces. Most of them are valid in greater generality: for all rational surfaces of positive degree.

\begin{predl}
\label{prop_loslo}
Let $D$ be an $r$-class on a weak del Pezzo surface~$X$ of degree $d$. Then Table~\ref{table_loslo} gives necessary and sufficient conditions for $D$ being left-orthogonal and strong left-orthogonal. 
In particular, if $D^2\ge -1$ then $D$ is
left-orthogonal if and only if $D$ is strong left-orthogonal.

More generally, let $X$ be a rational surface of degree $d>0$ and $D$ be an $r$-class on $X$. Then criteria of left-orthogonality from Table~\ref{table_loslo} hold, and for $r\le -2$ criteria of strong left-orthogonality hold.
\begin{table}[h]
\caption{Criteria of left-orthogonality and strong left-orthogonality of divisors on rational surfaces of positive degree. Criteria marked with $*$ are valid only for weak del Pezzo surfaces.}
\label{table_loslo}

\begin{center}
\begin{tabular}{|c|c|c|}
		 \hline
			$r$ & $D$ is left-orthogonal? & $D$ is strong left-orthogonal?\\
		 \hline
			$\le -3$ & iff $h^0(-D)=0$ & no \\
		 \hline
			$-2$ & iff $h^0(-D)=0$ & iff $h^0(D)=h^0(-D)=0$ \\
		 \hline	
			$-1\le r \le d-3$ & yes & yes $^*$ \\
		 \hline	
			$\ge d-2$ & iff $h^0(K_X+D)=0$ & iff $h^0(K_X+D)=0\,^*$ \\
		 \hline	
\end{tabular}
\end{center}
\end{table}
\end{predl}

For the proof we will need some lemmas.
\begin{lemma}
\label{lemma_Kpositive}
Let $X$ be  a rational surface with $\deg (X)=K_X^2>0$. Then $-K_X$ is effective.
\end{lemma}
\begin{proof}
We have $h^2(-K_X)=h^0(2K_X)=0$ by Serre duality since $X$ is rational. By Riemann-Roch formula we have $\chi(-K_X)=h^0(-K_X)-h^1(-K_X)=K_X^2+1\ge 2$. It follows that 
$h^0(-K_X)\ge 2$ and thus $-K_X$ is effective. 
\end{proof}

\begin{lemma}
\label{lemma_h0h2}
Let $X$ be  a rational surface with  $-K_X$ effective.
Let $D$ be a divisor on $X$ with $h^0(-D)=0$. Then $h^2(D)=0$.
\end{lemma}
\begin{proof}
We have $h^2(D)=h^0(K_X-D)$ by Serre duality. Since $-K_X$ is effective, we have an 
embedding
$$H^0(X,\O_X(K_X-D))\subset H^0(X,\O_X(-D))=0,$$
hence $h^0(K_X-D)=0$.
\end{proof}

\begin{lemma}
\label{lemma_sqrt}
Let $X$ be  a rational surface with  $\deg(X)>0$ and $D$ be an $r$-class on $X$ with $r\ge -1$. Then
\begin{enumerate}
\item $h^0(-D)=0$,
\item $h^0(D)>0$.
\end{enumerate}
\end{lemma}
\begin{proof}
(1) First we prove the statement for $X$ being a blow-up of $\P^2$ at several (maybe infinitesimal) points. Let $L,E_1,\ldots,E_n$ be the standard basis in $\Pic(X)$. Since $\deg(X)>0$, we have $n\le 8$. Let 
$D=aL+\sum_{i=1}^nb_iE_i$. 

\textbf{Claim.} We have $D=E_i$ for some $i$ or $a>0, a+b_i\ge 0$ for all $i$.

To prove the claim, assume that $r=-1$. Then all $(-1)$-classes in $\Pic(X)$ are listed in \cite[Prop. 26.1]{Ma} and the claim follows. Now assume $r\ge 0$. Then by Proposition~\ref{prop_D2} we have 
\begin{equation}
\label{eq_abbb}
D^2=a^2-\sum b_i^2\ge 0\quad\text{and}\quad D\cdot(-K_X)=3a+\sum b_i\ge 2.
\end{equation}
It follows that (using an inequality between arithmetic and quadratic means)
$$|a|\ge \sqrt{\sum_{i=1}^n b_i^2}\ge \frac{\sum_{i=1}^nb_i}{\sqrt{n}}\ge\frac{2-3a}{\sqrt{n}}\ge \frac{2-3a}{\sqrt{8}}$$
and $$\sqrt{8}|a|+3a\ge 2.$$
Therefore $a>0$. Finally, $a+b_i\ge 0$ is guaranteed by the first inequality in \eqref{eq_abbb}.

The Claim implies $-D$ cannot be effective: otherwise the image $-aL$ of $-D$ on $\P^2$ is effective.

Now we consider the general case: assume $X$ is a blow-up of some Hirzebruch surface $\mathbb F_s$ at several points. Moreover, unless $X=\mathbb F_s$ we can assume that $s=2m+1$ is odd. Let $B,F\subset X$ be the preimages of the $(-s)$-curve and the fiber respectively, we have $B^2=-s, B\cdot F=1, F^2=0$. Denote by 
$B,F,E_1,\ldots,E_n$ the standard basis in $\Pic(X)$. Since $\deg (X)>0$ we have $n\le 7$. Consider a blow-up $Y$ of $\P^2$ at $n+1$ points,  denote by $L,E'_0,E'_1,\ldots,E'_{n}$ the standard basis in $\Pic(Y)$. Define the linear map
$$\phi\colon \Pic(Y)\to \Pic(X)$$
be the rule
$$\phi(L)=B+(m+1)F,\quad \phi(E'_0)=B+mF,\quad \phi(E'_i)=E_i\quad\text{for}\quad i=1,\ldots,n.$$
The reader is welcome to check that $\phi$ is an isometry and maps $K_Y$ to $K_X$. 
Thus $\phi$ identifies $r$-classes on $Y$ and $X$. Let $D':=\phi^{-1}(D)$ be the $r$-class on $Y$, we can write $D'=aL+\sum_{i=0}^n b_iE'_i$. By the above Claim, we have either $D'=E'_i$ for some $i$ or $a>0$ and $a+b_0\ge 0$. Hence, $D=\phi(D')=E_i$, $D=B+mF$ or
$$D=a(B+(m+1)F)+b_0(B+mF)+\sum_{i=1}^nb_iE_i.$$
In the first two cases we readily have $h^0(-D)=0$, in the latter case we have
$$D=(a+b_0)B+(a(m+1)+b_0m)F+\sum_{i=1}^nb_iE_i.$$
The coefficient at $F$ equals to $(a+b_0)m+a$ and is positive by our assumptions.  
The divisor $S=B+sF$ is the class of a positive $s$-curve, hence $S$ is nef. We have 
$S\cdot (-D)=-(a(m+1)+b_0m)<0$, hence  $-D$ is not effective.

It remains to consider the case when $X$ is a Hirzebruch surface $\mathbb F_s$ with even $s$, we leave it to the reader.

(2) By Lemma~\ref{lemma_Kpositive}, $-K_X$ is effective. By part (1) we get $h^0(-D)=0$ and by Lemma~\ref{lemma_h0h2} we get $h^2(D)=0$. Now using Proposition~\ref{prop_D2} we get
$$h^0(D)\ge h^0(D)-h^1(D)=\chi(D)=D^2+2\ge 1.$$
\end{proof}

\begin{lemma}
\label{lemma_irred}
Let $X$ be a weak del Pezzo surface of degree $d$. Then 
a general  anticanonical divisor $Z\in |-K_X|$ is irreducible. 
\end{lemma}
\begin{proof}
Recall that by  \cite[Theorem 8.3.2]{Do}, the linear system $|-K_X|$ has no fixed components. If $d>1$ then $\dim|-K_X|=d>1$ and $Z$ is irreducible by Bertini's theorem. For $d=1$ we have
$\dim|-K_X|=1$ and argue as follows. Let $X\xra{f} \P^2$ be the blow-up of  points $P_1,\ldots,P_{8}$. To prove that a general divisor $Z\in |-K_X|=|3L-\sum E_i|$ is irreducible, it suffices to check that 
a general divisor in the linear system $|3L-\sum P_i|$ on $\P^2$ is irreducible. Note that any reducible divisor $Z\in |3L-\sum P_i|$ is a sum of a line $l_Z\in |L-\sum_{j\in J}P_j|$ and a conic $q_Z\in |2L-\sum_{j\in \{1,\ldots,8\}\setminus J}P_j|$ for some subset $J\subset\{1,\ldots,8\}$. If $l_Z$ or $q_Z$ is movable in the corresponding linear system then (as $\dim|-K_X|=1$) we get 
$$|-K_X|=|L-\sum_{j\in J}P_j|\times |2L-\sum_{j\in \{1,\ldots,8\}\setminus J}P_j|.$$
It follows that  $|-K_X|$ has a fixed part, what contradicts to \cite[Theorem 8.3.2]{Do}.
Otherwise $Z$ is uniquely defined by $J$, and there are only finitely many such reducible divisors. So the statement follows.
\end{proof}

\begin{proof}[Proof of Proposition~\ref{prop_loslo}]
First, we find criteria of left-orthogonality. We assume that $X$ is a rational surface with $\deg (X)>0$.
Since $D$ is numerically left-orthogonal, we have $h^0(-D)-h^1(-D)+h^2(-D)=0$. Therefore $D$ is left-orthogonal if and only if $h^0(-D)=h^2(-D)=0$. Suppose $r\ge -1$, then  $h^0(-D)=0$ by Lemma~\ref{lemma_sqrt}(1). Suppose $r\le d-3$, denote $D'=-K_X-D$. Then
by Lemma~\ref{lemma_dprime} $D'$ is an $r'$-class where $r'=d-4-r\ge -1$. By Lemma~\ref{lemma_sqrt}(1) we have $h^0(-D')=0$ and by Serre duality we have 
$h^2(-D)=h^0(K_X+D)=h^0(-D')=0$.
Now the statement follows.

Now we prove criteria of strong left-orthogonality. As before, $X$ is any rational surface with $\deg X>0$.
For a strong left-orthogonal divisor $D$ one has $D^2=\chi(D)-2=h^0(D)-2\ge -2$, therefore divisors with $r\le -3$ are not strong left-orthogonal. 

Let $r=-2$. Suppose $D$ is left-orthogonal, then $h^2(D)=0$ by Lemmas~\ref{lemma_Kpositive} and \ref{lemma_h0h2}. Since $D$ is numerically left-orthogonal and $\chi(D)=D^2+2=0$, $D$ is strong left-orthogonal iff $h^0(D)=0$. The statement for $r=-2$ follows. 

From now on we assume that $X$ is a weak del Pezzo surface.
It remains to demonstrate that a left-orthogonal divisor $D$ with $D^2\ge -1$ is strong left-orthogonal. We know by Lemma~\ref{lemma_h0h2} that $h^2(D)=0$.

By Lemma~\ref{lemma_irred}, we can choose an irreducible divisor $Z\in |-K_X|$.
Consider the standard exact sequence
$$0\to \O_X(K_X)\to \O_X\to \O_Z\to 0.$$
Twisting it by $\O_X(-D)$, we get the exact sequence
\begin{equation}
\label{eq_loslo}
0\to \O_X(K_X-D)\to \O_X(-D)\to \O_Z(-D)\to 0.
\end{equation}
The long exact sequence of cohomology associated with \eqref{eq_loslo} yields that $h^1(\O_X(K_X-D))=h^0(\O_Z(-D))$. Since $(-D)\cdot Z=D\cdot K_X=-2-D^2<0$ and $Z$ is irreducible, we have $h^0(\O_Z(-D))=0$ and by Serre duality $h^1(D)=h^1(K_X-D)=0$.
\end{proof}

The following consequences and variations of Proposition~\ref{prop_loslo} will be needed below.
\begin{lemma}
\label{lemma_effslo}
(1) Let $X$ be a rational surface with $-K_X$ effective. Then
$$R(X)=R^{\rm eff}(X)\sqcup(-R^{\rm eff}(X))\sqcup R^{\rm slo}(X).$$
(2) Let $X$ be a rational surface with $\deg (X)>0$. Then 
$$R^{\rm lo}(X)=R^{\rm slo}(X)\sqcup R^{\rm eff}(X).$$
\end{lemma}
\begin{proof}
(1) Note that $D\in R(X)\Longleftrightarrow -D\in R(X)$, also note that  $\chi(D)=\chi(-D)=0$ for any $(-2)$-class $D$. Hence $D$ is strong left-orthogonal iff $h^i(D)=h^i(-D)=0$ for $i=0,2$. 
Now assume $D\in R(X)$ and  $h^0(D)=h^0(-D)=0$. By Lemma~\ref{lemma_h0h2} applied to $D$ and $-D$ we get that $h^2(D)=h^2(-D)=0$. It follows that 
\begin{align*}
R^{\rm eff}(X) &= \{D\in R(X)\mid h^0(D)>0\},\\
-R^{\rm eff}(X) &= \{D\in R(X)\mid h^0(-D)>0\},\\
R^{\rm slo}(X) &= \{D\in R(X)\mid h^0(D)=h^0(-D)=0\},\\
\end{align*}
what concludes the statement.

(2) Follows from Lemma~\ref{lemma_Kpositive},  Proposition~\ref{prop_loslo} and part (1).
\end{proof}

\begin{corollary}
\label{cor_minustwoondP}
Any $(-2)$-class on a del Pezzo surface is strong left-orthogonal.
\end{corollary}
\begin{proof}
Use Lemma~\ref{lemma_effslo}(1). Indeed, any effective $(-2)$-class is a sum of classes of $(-2)$-curves by Proposition~\ref{prop_roots} and there are no  such curves on a del Pezzo surface by Lemma~\ref{lemma_negcurve}.
\end{proof}

We finish this section with an example demonstrating that (some part of) Proposition~\ref{prop_loslo} fails if~$X$ is only assumed to be of positive degree.
\begin{example}
Let $X=\mathbb F_3$ be a Hirzebruch surface with a $(-3)$-curve $B$ and a fiber $F$. We have $K_X^2=8>0$ but $X$ is not a weak del Pezzo surface. Consider the divisor $D=B+F$ on $X$, it is a $(-1)$-class. Then $D$ is left-orthogonal by Proposition~\ref{prop_loslo}, but not strong left-orthogonal: $\chi(D)=1$ while $h^0(D)=h^0(F)=2$. 
\end{example}

\section{Toric systems and admissible sequences}
\label{section_prelim2}

\subsection{Exceptional objects}

We start with recalling some concepts related with exceptional collections in triangulated categories. An object $\EE$  in the derived category $\D^b(\coh(X))$ of coherent sheaves on $X$ is said to be \emph{exceptional} if $\Hom^i(\EE,\EE)=0$ for
 all $i\ne 0$ and  $\Hom(\EE,\EE)=\k$. An ordered collection $(\EE_1,\ldots,\EE_n)$ of exceptional objects is said to be \emph{exceptional} if $\Hom^i(\EE_l,\EE_k)=0$ for
 all $i$ and $k<l$. An exceptional collection $(\EE_1,\ldots,\EE_n)$  is said to be \emph{strong exceptional} if $\Hom^i(\EE_l,\EE_k)=0$ for
 all $i\ne 0$ and all $k,l$. 
 An ordered collection $(\EE_1,\ldots,\EE_n)$ of objects in $\D^b(\coh(X))$ is said to be \emph{numerically exceptional} if for the Euler form $\chi(\EE,\FF):=\sum_i (-1)^i \dim\Hom^i(\EE,\FF)$ one has: 
 $$\chi(\EE_k,\EE_k)=1, \quad \chi(\EE_l,\EE_k)=0\quad\text{for }\quad k<l.$$
 A collection $(\EE_1,\ldots,\EE_n)$ is said to be \emph{full} if objects  $\EE_1,\ldots,\EE_n$ generate $\D^b(\coh X)$ as triangulated category. A collection $(\EE_1,\ldots,\EE_n)$ is said to be \emph{of maximal length } if the classes of  objects  $\EE_1,\ldots,\EE_n$ generate $K_0(X)_{num}$, the Grothendieck group of $X$ modulo numeric equivalence. Recall that for a rational surface $X$ one has
$$\rank K_0(X)=\rank\Pic(X)+2=12-\deg(X).$$

Any line bundle on a rational surface is exceptional. Therefore an ordered collection of line bundles $(\O_X(D_1),\ldots,\O_X(D_n))$ on a rational surface $X$ is exceptional (resp. strong exceptional) if and only if the divisor $D_l-D_k$ is left-orthogonal (resp. strong left-orthogonal) for all $k<l$.

Clearly, any full exceptional collection has maximal length. In general, the converse in not true: there are examples of surfaces of general type (classical Godeaux surface, \cite{BBS} or Barlow surface, \cite{BBKS}) possessing an exceptional collection of maximal length which is not full. But for rational surfaces there are no such examples known. For weak del Pezzo surfaces of degree $\ge 2$, it follows from a result by Sergey Kuleshov \cite[Theorem 3.1.8]{Ku} that any exceptional collection of vector bundles of maximal length is full.

\subsection{Toric systems}
Next we recall the important notion of a \emph{toric system}, introduced by Hille and Perling in \cite{HP}.

For a sequence $(\O_X(D_1),\ldots,\O_X(D_n))$ of line bundles one can consider 
the infinite sequence (called a \emph{helix}) $(\O_X(D_i)), i\in \Z$, defined by the rule
$D_{k+n}=D_k-K_X$. From Serre duality it follows that the collection $(\O_X(D_1),\ldots,\O_X(D_n))$ is exceptional (resp. numerically exceptional) if and only if any collection of the form $(\O_X(D_{k+1}),\ldots,\O_X(D_{k+n}))$ is exceptional (resp. numerically exceptional). One can consider the $n$-periodic sequence $$A_k=D_{k+1}-D_k$$ 
of divisors on $X$. Following Hille and Perling, we will consider the finite sequence $(A_1,\ldots,A_n)$ with the cyclic order and will treat the index $k$ in $A_k$ as a residue modulo $n$.  Vice versa, for any sequence $(A_1,\ldots,A_n)$ one can construct the infinite sequence $(\O_X(D_i)), i\in \Z$, with the property $D_{k+1}-D_k=A_{k\mod n}$. 
Then the corresponding collection of line  bundles would be
\begin{multline}
\label{eq_DDD}
(\O_X(D_1),\O_X(D_2),\ldots,\O_X(D_n))=\\
=(\O_X(D_1),\O_X(D_1+A_1),\O_X(D_1+A_1+A_2),\ldots,\O_X(D_1+A_1+\ldots+A_{n-1})).
\end{multline}

\begin{definition}[See {\cite[Definitions 3.4 and 2.6]{HP}}]
\label{def_ts}
Let $X$ be a smooth projective surface. 
A sequence $(A_1,\ldots,A_n)$ in $\Pic(X)$ is called a \emph{toric system} if $n=\rank K_0(X)_{num}$ and the following conditions are satisfied (where indices are treated modulo $n$):
\begin{enumerate}
\item $A_iA_{i+1}=1$;
\item $A_iA_{j}=0$ if $j\ne i,i\pm 1$;
\item $A_1+\ldots+A_n=-K_X$.
\end{enumerate}
\end{definition}

Note that a cyclic shift $(A_k,A_{k+1},\ldots,A_n,A_1,\ldots,A_{k-1})$ of a toric system $(A_1,\ldots,A_n)$  is also a toric system. Also, note that by our definition any toric system has maximal length.

\begin{example}[See {\cite[Section 3.1]{HI}}]
Let $Y$ be a smooth projective toric surface. Its torus-invariant prime divisors form a cycle, denote them $T_1,\ldots,T_n$ in the cyclic order. Then $(T_1,\ldots,T_n)$ is a toric system on~$Y$. 
\end{example}

The following is proved in \cite[Lemma 3.3]{HP}, see also \cite[Propositions 2.8 and 2.15]{EL}.
\begin{predl}
\label{prop_tsne}
Let $X$ be a surface with $\chi(\O_X)=1$. Then a sequence $(A_1,\ldots,A_n)$ in $\Pic(X)$ is a toric system if and only if the corresponding collection \eqref{eq_DDD} is numerically exceptional of maximal length for any $D_1$. In particular, if $A$ is a toric system then $A_i^2+2=-A_i\cdot K_X$ for any $i$.
\end{predl}

\begin{definition}
\label{def_exts}
Let $X$ be a surface such that the sheaf $\O_X$ is exceptional (that is, $H^0(\O_X)=\k$, $H^1(\O_X)=H^2(\O_X)=0$).
A toric system $(A_1,\ldots,A_n)$ on $X$ is called \emph{exceptional} (resp. \emph{strong exceptional}) if the corresponding  collection $(\O_X(D_1),\ldots,\O_X(D_n))$ is  exceptional (resp. strong exceptional). Note that exceptional toric systems are stable under cyclic shifts while strong exceptional toric systems are not in general.

A toric system $(A_1,\ldots,A_n)$ on $X$ is called \emph{cyclic strong exceptional} if the corresponding collection 
$$(\O_X(D_{k+1}),\ldots,\O_X(D_{k+n}))$$ 
is strong exceptional for any $k\in \Z$. Equivalently: if all cyclic shifts $$(A_k,A_{k+1},\ldots,A_n,A_1,\ldots,A_{k-1})$$
are strong exceptional.
\end{definition}

\begin{definition}
\label{def_kl}
We will use the following notation: let $k,l\in [1,\ldots,n]$ and $k\not\equiv l+1\pmod{n}$. By a \emph{cyclic segment} $[k,\ldots,l]\subsetneq[1,\ldots,n]$
we mean the following set of indices:
\begin{equation*}
[k,\ldots,l]:=
\begin{cases}[k,k+1,\ldots,l-1,l], & \text{if}\quad k\le l,\\
[k,k+1,\ldots,n,1,2,\ldots,l], & \text{if}\quad k>l. 
\end{cases}
\end{equation*}
Note that $[1,\ldots,n]$ is not a cyclic segment of $[1,\ldots,n]$.
\end{definition}

For a toric system $(A_1,\ldots,A_n)$ and $k,l \in [1,\ldots,n]$ denote 
$$A_{k\ldots l}:=\sum_{i\in [k,\ldots, l]}A_i. $$

\begin{predl}
\label{prop_Akl2}
Let $X$ be a surface with $\chi(\O_X)=1$.
For any toric system $A$ on $X$, the divisor
$A_{k\ldots l}$ is numerically left-orthogonal with 
\begin{equation*}
A_{k\ldots l}^2+2=\sum_{i\in [k,\ldots, l]}(A_i^2+2).
\end{equation*}
\end{predl}
\begin{proof}
Follows from Propositions~\ref{prop_classes0} and~\ref{prop_tsne}. 
\end{proof}

\begin{corollary}
\label{cor_minustwo}
Let $X$ be a surface with $\chi(\O_X)=1$ and $A$ be a toric system on $X$.
If one has $A_i^2\ge -2$ for all $i$, then one has $A_{k\ldots l}^2\ge -2$ for any cyclic segment $[k,\ldots,l]\subset  [1,\ldots, n]$. 
\end{corollary}

\begin{predl}
\label{prop_cyclic1type}
Let $X$ be a surface such that the sheaf $\O_X$ is exceptional. 
Then for any strong exceptional toric system $(A_1,\ldots,A_{n})$ on $X$ we have $A_i^2\ge -2$ for all $i=1,\ldots,n-1$ and
for any cyclic strong exceptional toric system $(A_1,\ldots,A_{n})$ on $X$ we have $A_i^2\ge -2$ for all $i=1,\ldots,n$.
\end{predl}
\begin{proof}
The divisor $A_i$ under consideration is strong left-orthogonal. Hence by Proposition~\ref{prop_D2} $A_i^2=\chi(A_i)-2=h^0(A_i)-2\ge -2$.
\end{proof}

The next proposition is a straightforward consequence of definitions.
\begin{predl}
\label{prop_straightforward}
Let $X$ be  a surface such that the sheaf $\O_X$ is exceptional.
\begin{enumerate}
\item
A toric system $(A_1,\ldots,A_n)$ on $X$ is  exceptional if and only if the divisor $A_{k\ldots l}$ is left-orthogonal for all $1\le k<l\le n-1$ and if and only if the divisor $A_{k\ldots l}$ is left-orthogonal for any cyclic segment  $[k,\ldots,l]\subset [1,\ldots, n]$.
\item
A toric system $(A_1,\ldots,A_n)$ on $X$ is strong  exceptional if and only if the divisor $A_{k\ldots l}$ is strong left-orthogonal for all $1\le k<l\le n-1$.
\item
A toric system $(A_1,\ldots,A_n)$ on $X$ is  cyclic strong exceptional if and only if the divisor $A_{k\ldots l}$ is strong  left-orthogonal for any
cyclic segment  $[k,\ldots,l]\subset [1,\ldots, n]$.
\end{enumerate}
\end{predl}

For weak del Pezzo surfaces we prove handy criteria for exceptionality of toric systems in terms of effectiveness of some divisors.

\begin{theorem}
\label{theorem_checkonlyminustwo}
Let $A=(A_1,\ldots,A_n)$ be a toric system on a weak del Pezzo surface $X$ of degree $d$. Suppose $A^2_i\ge -2$ for  $1\le i\le n-1$. 
Then: 
\begin{enumerate}
\item
$A$ is exceptional if and only if the following two conditions hold: 
\begin{enumerate}
\item for any $1\le k\le l\le n-1$ such that   $A_k^2=\ldots=A_l^2=-2$, the divisor $-A_{k\ldots l}$ is not effective,
\item if $A_n^2\le -2$ then for any $0\le l<k\le n$ such that 
$$A_k^2=\ldots=A_{n-1}^2=A_1^2=\ldots=A_l^2=-2$$  
the divisor $-A_{k\ldots n\ldots l}$ is not effective (if $l=0$ this reads as $-A_{k\ldots n}$).
\end{enumerate}
\item
$A$ is strong exceptional if and only if $A$ is exceptional and the following holds: for any $1\le k\le l\le n-1$ such that $A_k^2=\ldots=A_l^2=-2$, the divisors $A_{k\ldots l}$ and $-A_{k\ldots l}$ are not effective.
\end{enumerate}

Suppose moreover that $A^2_i\ge -2$ for  all $i$. Then
\begin{enumerate}
\item[(3)]
$A$ is exceptional if and only if the following holds: for any cyclic segment $[k,\ldots, l]\subset [1,\ldots, n]$ such that $A_i^2=-2$ for all $i\in [k,\ldots,l]$, the  divisor $-A_{k\ldots l}$ is not effective.
\item[(4)]
$A$ is cyclic strong exceptional if and only if the following holds: for any cyclic segment $[k,\ldots, l]\subset [1,\ldots, n]$ such that $A_i^2=-2$ for all $i\in [k,\ldots,l]$, the divisors $A_{k\ldots l}$ and $-A_{k\ldots l}$ are not effective.
\end{enumerate}
\end{theorem} 
\begin{proof}
(1) By Proposition \ref{prop_straightforward}, $A$ is exceptional if and only if for any $1\le k\le l\le n-1$ the divisor $A_{k\ldots l}$ is left-orthogonal.
Furthermore, if $D$ is left-orthogonal then $-D$ is
not effective by definition. This proves the ``only if'' part.

To prove the ``if'' part take any $1\le  k \le  l \le n-1$, set
$$D = A_{k\ldots l}\quad \text{and} \quad D' = -D -K_X = A_{l+1\ldots k-1}.$$
Note that $D^2 + (D')^2 = d - 4$ by Lemma~\ref{lemma_dprime}. Now, $D^2\ge -2$ by Proposition~\ref{prop_Akl2}. If $D^2=-2$ then $D$ is left-orthogonal by assumption (a) and Proposition~\ref{prop_loslo}; if $-1 \le  D^2 \le  d-3$ then $D$
is left-orthogonal by Proposition~\ref{prop_loslo}. Finally, if $D^2\ge  d-2$ then 
$(D')^2\le -2$. We claim that $-D'$ is not effective.
One can decompose $D'$ as
$$D'=A_{l+1\ldots n\ldots k-1}=A_{l+1\ldots l_1-1}+A_{l_1\ldots n\ldots k_1}+A_{k_1+1\ldots k-1}=:D_-+D''+D_+,$$
such that the following holds:
\begin{enumerate}
\item $A_{l_1}^2=\ldots =A_{n-1}^2=A_1^2=\ldots=A_{k_1}^2=-2$, 
\item if $l_1\ne l+1$ (i.e., $D_-\ne 0$) then $A_{l_1-1}^2\ge -1$,
\item if $k_1\ne k-1$ (i.e., $D_+\ne 0$) then $A_{k_1+1}^2\ge -1$.
\end{enumerate}
From Proposition~\ref{prop_Akl2} and our assumptions ($A_i^2\ge -2$ for $i=1,\ldots,n-1$) it follows that 
$(D'')^2=A_n^2$, $A_n^2\le (D')^2\le -2$ and 
$(D_-)^2\ge -1$ (or $D_-=0$), $(D_+)^2\ge -1$ (or $D_+=0$).
Note that $D_-,D_+$ are effective by Lemma~\ref{lemma_sqrt}(2).
Assume now that $-D'$ is effective, then $-D''=-D'+D_-+D_+$ is effective which contradicts to  assumption (b). We conclude that $-D'=K_X+D$ is not effective and hence $D$  is left-orthogonal by Proposition~\ref{prop_loslo}.  Thus, in all cases $D$ is left-orthogonal and therefore $A$ is exceptional.

(2) Suppose $1\le k\le l\le n-1$,  then by (1) $D=A_{k\ldots l}$ is left-orthogonal. By Proposition~\ref{prop_Akl2}, one has $D^2\ge -2$ and $D^2=-2$ if and only if $A_k^2=\ldots=A_l^2=-2$. If $D^2=-2$,  then $D$ is strong left-orthogonal by assumptions and Proposition~\ref{prop_loslo}, if $D^2>-2$ then $D$ is strong left-orthogonal by Proposition~\ref{prop_loslo}.

(3) follows from (1), (4) is similar to (2).
\end{proof}

\begin{corollary}
Let $A=(A_1,\ldots,A_n)$ be a strong exceptional toric system  with $A_n^2\ge -1$ on a weak del Pezzo surface. Then $A$ is cyclic strong exceptional.
\end{corollary}

\begin{proof}
By Proposition~\ref{prop_loslo}, we have $A_i^2\ge -2$ for any $1\le i\le n-1$ since such $A_i$ is strong left-orthogonal. Also $A_n^2\ge -1> -2$.
Use Theorem~\ref{theorem_checkonlyminustwo}(4). Clearly, any cyclic segment $[k,\ldots,l]\subset [1,\ldots,n]$ such that $A_k^2=\ldots=A_l^2=-2$ lies in $[1,\ldots,n-1]$. Therefore $\pm A_{k\ldots l}$ are not effective by  Theorem~\ref{theorem_checkonlyminustwo}(2).
\end{proof}

\begin{corollary}
Any toric system $(A_1,\ldots,A_n)$  with $A_i^2\ge -2$ for all $i$ on a del Pezzo surface is cyclic strong exceptional. 
\end{corollary}
\begin{proof}
Use Theorem~\ref{theorem_checkonlyminustwo}(4) and Corollary \ref{cor_minustwoondP}.
\end{proof}

We finish this section with a very important theorem which is due to Hille and Perling \cite[Theorem 3.5]{HP} for rational surfaces  and to Charles Vial  \cite[theorem 3.5]{Vi} for the general case.
\begin{theorem}
\label{theorem_HP}
Let $A=(A_1,\ldots,A_n)$ be a toric system on a smooth projective surface~$X$ such that $\chi(\O_X)=1$ ($X$ is not necessarily rational). Then there exists a smooth projective toric surface~$Y$ with torus-invariant divisors $T_1,\ldots,T_n$ such that $A_i^2=T_i^2$ for all $i$.
\end{theorem}

We remark that a toric system $A$ is uniquely defined by the sequence $A^2$ in the following sense: if $A=(A_1,\ldots,A_n),B=(B_1,\ldots,B_n)$ are two toric systems on rational  surfaces $X$ and $Y$ respectively and $A_i^2=B_i^2$ for all $i=1,\ldots, n$ then there exists an isomorphism $\Pic(X)\to\Pic(Y)$ preserving the intersection form and mapping $A_i$ to $B_i$ (and thus $K_X$ to $K_Y$). To see this, one should  consider $\Z^n$ with the bilinear form $q$ defined in the standard basis $(e_1,\ldots,e_n)$ by
$$q(e_i, e_j) :=
\begin{cases} A_i^2, & i=j;\\
1, & \text{if}\quad \mid i-j\mid=1;\\
0, & \text{otherwise.}
\end{cases}
$$
There is a surjective map $\alpha\colon \Z^n\to \Pic(X)$ sending $e_i$ to $A_i$. Moreover, $\alpha$ is compatible with $q$ and the intersection form. Therefore, the triple $(\Pic(X), \text{the intersection form}, (A_i))$ is identified with the triple $(\Z^n/\ker q, \text{the form induced by}\,\, q, (\im e_i))$, and the latter depends only on the sequence $A^2$.

\subsection{Admissible sequences}
\label{section_adm}

For a sequence $(a_1,\ldots,a_n)$ of integers we define the \emph{$m$-th elementary augmentation} as follows:

\begin{itemize}
\item $\augm_1(a_1,\ldots,a_n) = (-1,a_1-1,a_2,\ldots,a_{n-1},a_n-1)$;
\item $\augm_m(a_1,\ldots,a_n)=(a_1,\ldots,a_{m-2},a_{m-1}-1,-1,a_m-1,a_{m+1},\ldots,a_n)$ for $2\le m\le n$; 
\item $\augm_{n+1}(a_1,\ldots,a_n)=(a_1-1,a_2,a_3,\ldots,a_n-1,-1)$.
\end{itemize}

\begin{definition}
\label{def_adm}
We call a sequence  \emph{admissible} if it can be obtained from the sequence $(0,k,0,-k)$ or $(k,0,-k,0)$ where $k\in\Z$ by applying several elementary augmentation.
\end{definition}

It is not hard to see that admissible sequences are stable under cyclic shifts:
$$\sh(a_1,\ldots, a_n):=(a_2,\ldots,a_n,a_1)$$
and symmetries:
$${\rm sym}(a_1,\ldots, a_n):=(a_{n-1},a_{n-2},\ldots,a_1,a_n).$$
We define symmetry like this and not in a more natural way because such symmetry preserves the condition $a_i\ge -2$ for $i=1,\ldots,n-1$.
This condition describes admissible sequences coming from strong exceptional toric systems, see Theorem~\ref{theorem_checkonlyminustwo} or Proposition~\ref{prop_cyclic1type}. The analogous symmetry for toric systems corresponds to the operation which sends an exceptional collection of line bundles $(\O_X(D_1),\ldots,\O_X(D_n))$ to 
$(\O_X(-D_n),\ldots,\O_X(-D_1))$.

Operations ${\rm sh}$ and ${\rm sym}$ act on the set of admissible sequences of fixed length $n$. They generate the group isomorphic to the $n$-th dihedral group: the group of all permutations  of the index set $[1,\ldots,n]$  preserving or inverting the cyclic order. This group is isomorphic to $\Z/n\Z\rtimes \Z/2Z$.

For a toric system $A=(A_1,\ldots,A_n)$ on a surface $X$, we denote 
$$A^2:=(A_1^2,\ldots,A_n^2).$$

The motivation for considering admissible sequences is the following discovery due to Hille and Perling.
\begin{predl}
\label{prop_A2adm}
Let $X$ be a surface with $\chi(\O_X)=1$, suppose $X\not\cong \P^2$. Then for any toric system $A$ on $X$ the sequence $A^2$ is admissible.
\end{predl}
\begin{proof}
First, assume that $X$ is a toric surface with torus invariant divisors $T_1,\ldots,T_n$. Then the sequence $T^2=(T_1^2,\ldots,T_n^2)$ is admissible. 
Indeed, for $X$ a Hirzebruch surface one has  $n=4$ and the statement is clear. Otherwise $X$ has a torus-invariant $(-1)$-curve $E$ by \cite[Theorem 2.1]{HP}. Let $E=T_k$, and consider the blow-down $X'$ of $E$. The torus-invariant divisors on $X'$ are $T'_1,\ldots,T'_{k-1},T'_{k+1},\ldots,T'_n$, and $(T'_{k\pm 1})^2=T_{k\pm 1}^2+1$, $(T'_{i})^2=T_{i}^2$ otherwise. Therefore $T^2=\augm_k((T')^2)$, and we proceed by induction.
The general case follows from the toric case and Theorem~\ref{theorem_HP}.
\end{proof}

\begin{definition}
\label{def_1kind}
We say that an admissible sequence $(a_1,\ldots,a_n)$ is \emph{of the first kind} if $a_i\ge -2$ for all $1\le i\le n$. We say that a toric system $A$ is of \emph{the first kind}  if the sequence $A^2$ is of the first kind.
\end{definition}

By Proposition~\ref{prop_cyclic1type},  any cyclic strong exceptional toric system $A$ is of the first kind. One could say also that toric systems of the first kind are \emph{numerically cyclic strong exceptional}.

\begin{predl}
\label{prop_csadm}
All admissible sequences of the first kind are, up to cyclic shifts and symmetries, in  Table~\ref{table_csadm}. In particular, if a surface $X$ with $\chi(\O_X)=1$ has a toric system $(A_1,\ldots,A_n)$ of the first kind, then $n\le 9$ and $\deg(X)\ge 3$.

\begin{table}[h]
\caption{Admissible sequences of the first kind}
\begin{center}
\begin{tabular}{|c|c|}
		 \hline
			type & sequence\\
		 \hline
			$\P^1\times \P^1$ & $(0,0,0,0)$ \\
		 \hline	
			$\mathbb F_1$ & $(0,-1,0,1)$ \\
		 \hline	
			$\mathbb F_2$ & $(0,-2,0,2)$ \\
		 \hline	
			5a & $(0,0,-1,-1,-1)$ \\
		 \hline	
			5b & $(0,-2,-1,-1,1)$ \\
		 \hline	
			6a & $(-1,-1,-1,-1,-1,-1)$ \\
		 \hline	
			6b & $(-1,-1,-2,-1,-1,0)$ \\
		 \hline	
			6c & $(0,0,-2,-1,-2,-1)$ \\
		 \hline	
			6d & $(0,-2,-2,-1,-2,1)$ \\
		 \hline	
			7a & $(-1,-1,-1,-2,-1,-2,-1)$ \\
		 \hline	
			7b & $(-2,-1,-2,-2,-1,-1,0)$ \\
		 \hline	
			8a & $(-2,-1,-2,-1,-2,-1,-2,-1)$ \\
		 \hline	
			8b & $(-2,-1,-1,-2,-1,-2,-2,-1)$ \\
		 \hline	
			8c & $(-2,-1,-2,-2,-2,-1,-2,0)$ \\
		 \hline	
			9 & $(-2,-2,-1,-2,-2,-1,-2,-2,-1)$ \\
		 \hline	
\end{tabular}
\end{center}
\label{table_csadm}
\end{table}

\end{predl}

\begin{proof}
In the obvious notation we have
\begin{align*}
(5a)&=\augm_4(\P^1\times \P^1), & 
(5b)&=\augm_3(\mathbb F_1), & 
(6a)&=\augm_2(5a), & 
(6b)&=\augm_1(5b), \\
(6c)&=\augm_4(5a), & 
(6d)&=\augm_4(5b), & 
(7a)&=\augm_5(6a), & 
(7b)&=\augm_2(6b), \\
(8a)&=\augm_2(7a), & 
(8b)&=\augm_8(7a),& 
(8c)&=\augm_6(7b), & 
(9)&=\augm_3(8b).
\end{align*}
The reader is welcome to check that all other augmentations of the listed sequences are either not of the first kind or contained in the table up to some shifts or/and symmetries.
\end{proof}

We remark that the
admissible sequences are exactly the sequences of self-intersection numbers of torus-invariant divisors on toric surfaces. Moreover, by \cite[Prop. 2.5]{HP} the sequence is of the first kind if and only if the corresponding toric surface has nef anti-canonical class. Therefore our Table~\ref{table_csadm} coinsides with Table 1 in \cite{HP} which lists such surfaces.

\section{Operations with toric systems}
\label{section_prelim3}

\subsection{Transpositions and shifts}

Let $(\O_X(D_1),\ldots,\O_X(D_n))$ be an exceptional collection of line bundles on a surface $X$. Recall that the line bundles 
$$O_X(D_k),\O_X(D_{k+1}),\ldots,\O_X(D_l)$$ 
\emph{form a block} if they are completely orthogonal to each other: $\Hom^i(\O_X(D_p),\O_X(D_q))=0$ for any $k\le p<q\le l$ and $i$.
One can reorder bundles in a block and get essentially the same exceptional collection.
Note that for mutually orthogonal $\O_X(D_k),\O_X(D_{k+1})$ the difference $D_{k+1}-D_k=A_k$ is a strong left-orthogonal divisor with $\chi(A_k)=0$. Hence $A_k^{2}=-2$ by Proposition~\ref{prop_D2}. Similarly, for a block of orthogonal line bundles
$O_X(D_k),\O_X(D_{k+1}),\ldots,\O_X(D_l)$ one has a chain $A_k,\ldots,A_{l-1}$ of strong left-orthogonal $(-2)$-classes.
 
On the side of toric systems, transposition of two completely orthogonal line bundles results in the following operation. 
\begin{definition}
Let $A=(A_1,\ldots,A_n)$ be a toric system. Suppose $A_k^2=-2$ for some $k$. Denote
\begin{align*}
\perm_1(A)&=(-A_1,A_1+A_2,A_3,\ldots,A_{n-1},A_n+A_1)\quad\text{if}\quad k=1,\\
\perm_k(A)&=(A_1,\ldots,A_{k-2},A_{k-1}+A_k,-A_k,A_{k}+A_{k+1},A_{k+2},\ldots,A_n)\quad\text{if}\quad 1<k<n,\\
\perm_n(A)&=(A_1+A_n,A_2,\ldots,A_{n-2},A_{n-1}+A_n,-A_n)\quad\text{if}\quad k=n.
\end{align*}
\end{definition}
It is easy to see that $\perm_k(A)$ is also a toric system. If $A$ corresponds to the  numerically exceptional collection $(\O_X(D_1),\ldots,\O_X(D_n))$ then $\perm_k(A)$ corresponds to the numerically exceptional collection 
\begin{align*}
&(\O_X(D_1),\ldots,\O_X(D_{k-1}),\O_X(D_{k+1}),\O_X(D_k),\O_X(D_{k+2}),\ldots,\O_X(D_n))& &\text{for}\quad 1\le k\le n-1,\\
&(\O_X(D_n+K_X),\O_X(D_2),\ldots,\O_X(D_{n-1}),\O_X(D_1-K_X))
& &\text{for}\quad  k=n.
\end{align*}
The operation $\perm_k$ is called $k$-th \emph{transposition}. It is an involution and $\perm_k,\ldots,\perm_l$ define an action of the symmetric group $S_{l-k+2}$ on the set of toric systems $A$ satisfying $A_k^2=\ldots=A_l^2=-2$. A composition of several transpositions will be called a \emph{permutation}. Note also that
\begin{equation}
\label{eq_perm2}
(\perm_k(A))^2=A^2.
\end{equation}

\begin{lemma}
\label{lemma_perm}
Let $X$ be a surface such that the sheaf $\O_X$ is exceptional. Suppose $A=(A_1,\ldots,A_n)$ is a toric system on $X$ and $A_k^2=-2$ for some $k$, $1\le k\le n$. 
\begin{enumerate}
\item If $A$ is exceptional then $\perm_k(A)$ is exceptional if and only if the divisor $A_k$ is strong left-orthogonal.
\item If $k\ne n$ then $A$ is strong exceptional if and only if $\perm_k(A)$ is strong exceptional.
\item  $A$ is cyclic strong exceptional if and only if $\perm_k(A)$ is cyclic strong exceptional.
\end{enumerate}
\end{lemma}
\begin{proof}
It follows easily from Proposition~\ref{prop_straightforward}. Note also that $A_k$ is strong left-orthogonal if and only if $-A_k$ is strong left-orthogonal. Indeed, both mean that $h^i(A_k)=h^i(-A_k)=0$ for $i=0,1,2$ as $\chi(A_k)=\chi(-A_k)=0$ for a $(-2)$-class.
\end{proof}

We introduce also the natural operation of cyclic shift:
$$\sh((A_1,\ldots,A_n))= (A_2,A_3,\ldots,A_n,A_1).$$
It corresponds to the following operation with collections of line bundles:
$$(\O_X(D_1),\ldots,\O_X(D_n))\Longrightarrow (\O_X(D_2),\ldots,\O_X(D_n),\O_X(D_{n+1})),$$
where $D_{n+1}=D_1-K_X$.

It follows from Proposition~\ref{prop_straightforward} that cyclic shift preserves exceptional and  cyclic strong exceptional toric systems but may not preserve strong exceptional toric systems.

\subsection{Augmentations}
\label{subsection_augm}

Following Hille and Perling \cite{HP}, we define augmentations. They provide a wide class of explicitly constructed toric systems. 

\begin{definition}
\label{def_elemaugm}
Let $A'=(A'_1,\ldots,A'_n)$ be a toric system on a surface $X'$, and let $p\colon X\to X'$ be the blow up of a point with the exceptional divisor $E\subset X$. Denote $A_i=p^*A'_i$. Then one has the following toric systems on $X$:
\begin{align*}
\augm_{p,1}(A')=&(E,A_1-E,A_2,\ldots, A_{n-1},A_n-E);\\
\augm_{p,m}(A')=&(A_1,\ldots,A_{m-2}, A_{m-1}-E,E,A_{m}-E,A_{m+1},\ldots,A_n)\quad\text{for}\quad 2\le m\le n;\\
\augm_{p,n+1}(A')=&(A_1-E,A_2,\ldots, A_{n-1},A_n-E,E).
\end{align*}
The toric systems $\augm_{p,m}(A')$ ($1\le m\le n+1$) are called   \emph{elementary augmentations} of the toric system~$A'$.
\end{definition}

\begin{lemma}
\label{lemma_bld1kind}
Augmentations of toric systems and sequences of integers are compatible: for any toric system $A=(A_1,\ldots,A_n)$ on a surface and any $1\le k\le n$ one has
$$(\augm_k(A))^2=\augm_k(A^2).$$
Moreover, if $\augm_k(A)$ is of the first kind (see Definition~\ref{def_1kind}) then $A$ also is.
\end{lemma}
\begin{proof}
Follows from the definitions.
\end{proof}

\begin{predl}[{See \cite[Proposition 3.3]{EL}}]
\label{prop_elemaugm}
In the notation of Definition~\ref{def_elemaugm}, let $A$ be a toric system on a surface $X$ such that $A_m=E$ for some $m$. Then $A=\augm_{p,m}(A')$ for some toric system $A'$ on $X'$.
\end{predl}

\begin{definition}
\label{def_augm1}
A toric system $A$ on a rational surface $X$ is called a \emph{standard augmentation} if~$X$ is a Hirzebruch surface or $A$ is an elementary augmentation of some standard augmentation. Equivalently: $A$ is a standard augmentation if there exists a chain of blow-ups
$$X=X_n\xra{p_n} X_{n-1}\to \ldots X_1\xra{p_1} X_0$$
where $X_0$ is a Hirzebruch surface and 
$$A=\augm_{p_n,k_n}(\augm_{p_{n-1},k_{n-1}}(\ldots \augm_{p_1,k_1}(A')\ldots ))$$
for some $k_1,\ldots,k_n$ and a toric system $A'$ on $X_0$. In this case we will say that $A$ is a \emph{standard augmentation along the chain $p_1,\ldots,p_n$}.
\end{definition}

\begin{remark}
To be more accurate, one  should add that (the unique) toric system $(L,L,L)$ on $\P^2$ is also considered as a standard augmentation. To simplify the forthcoming definitions and statements, we will ignore this issue.
\end{remark}

\begin{predl}
\label{prop_augmexc}
Let $X'$ be a surface such that the sheaf $\O_{X'}$ is exceptional, let $p\colon X\to X'$ be a blow-up of a point. Assume $A'$ is a toric system on $X'$ and $A=\augm_k(A')$. Then
\begin{enumerate}
\item $A$ is exceptional if and only if $A'$ is exceptional;
\item if $A$ is strong exceptional then $A'$ is strong exceptional;
\item if $A$ is cyclic strong exceptional then $A'$ is cyclic strong exceptional.
\end{enumerate}
\end{predl}
\begin{proof}
(1) and (2) are in  \cite[Prop. 2.21]{EL}. For (3), note that cyclic shifts commute with elementary augmentation and use (2) and Definition~\ref{def_exts}.
\end{proof}

An exceptional collection of line bundles is called a \emph{standard augmentation} if the associated toric system is a standard augmentation. One of main results in \cite{HP}
is that any strong exceptional collection of line bundles on a toric surface comes from a standard augmentation. But this ``comes from'' does not literally means ``is'' (as this is false, see Example~\ref{example_bad} below). Instead it means (see \cite[Theorem 8.1]{HP}) that such collection has a \emph{normal form} which is a standard augmentation. This normal form is obtained from the original collection by reordering bundles which are mutually orthogonal. In other words, the toric system associated to the normal form is obtained from the original toric system by transpositions. 

Instead of using normal forms we prefer to introduce the following two definitions that we find more convenient.

\begin{definition}
\label{def_augm2}
A toric system $A$ on a rational surface $X$ is called an \emph{augmentation in the weak sense} 
if~$A$ can be obtained from a toric system on a Hirzebruch surface by several transpositions, cyclic shifts and elementary augmentations  (applied in any order). 
\end{definition}

\begin{definition}
\label{def_augm4}
An exceptional (resp. strong exceptional, cyclic strong exceptional) toric system $A=(A_1,\ldots,A_n)$ on a rational surface $X$ is called an \emph{exceptional (resp. strong exceptional, cyclic strong exceptional) augmentation} if $X$ is a Hirzebruch surface or $A$ can be obtained by a transposition, cyclic shift or an elementary augmentation from some toric system which is an exceptional (resp. strong exceptional, cyclic strong exceptional) augmentation.
\end{definition}

\begin{remark}
If a toric system is a standard augmentation, it is an augmentation in the weak sense. 
\end{remark}

\begin{remark}
In the both above  definitions, a sequence 
$$A^{(0)}, A^{(1)}, \ldots, A^{(m)}=A$$ 
of toric systems is required to exist, where $A^{(0)}$ is a toric system on a Hirzebruch surface and any $A^{(k)}$ is obtained from $A^{(k-1)}$ by one of the operations: elementary augmentation, cyclic shift or transposition. Suppose $A$ is exceptional, then $A^{(k)}$-s are not automatically exceptional. The difference between Definitions~\ref{def_augm2} and \ref{def_augm4} is that in 
Definition~\ref{def_augm4} all intermediate toric systems $A^{(k)}$ are required to be exceptional while in 
Definition~\ref{def_augm2} they are not.
\end{remark}

Next Proposition follows from results of  \cite{EL}.
\begin{predl}
\label{prop_full}
Let $(\O_X(D_1),\ldots , \O_X(D_n))$ be an exceptional collection of line bundles of maximal length on a rational surface $X$. Suppose that the corresponding toric system $A$ is an exceptional augmentation. Then this collection is full.
\end{predl}
\begin{proof}
According to Definition~\ref{def_augm4}, there exists a sequence
$$A^{(0)}, A^{(1)}, \ldots, A^{(m)}=A$$ 
of exceptional toric systems such that $A^{(0)}$ is a system on a Hirzebruch surface and any $A^{(k)}$ is obtained from $A^{(k-1)}$ by a cyclic shift, a transposition or an elementary augmentation. Thus we have to check the following: 
\begin{enumerate}
\item any exceptional toric system on a Hirzebruch surface is full;
\item any cyclic shift of a full exceptional toric system is full;
\item any transposition of a full exceptional toric system is full;
\item any elementary augmentation of a full exceptional toric system is full.
\end{enumerate}
Propositions 2.16, 2.19 in \cite{EL} imply (1). Statement (2) holds because a cyclic shift of a full exceptional collection is the mutation of the first object to the last position. For (3), note that $\perm_k(A_1,\ldots,A_n)$ does not change the set of line bundles in the corresponding exceptional collections for $1\le k\le n-1$. For $k=n$ use (2) and the equality $\perm_n=\sh\circ\perm_1\circ\sh^{-1}$. Finally, (4) is 	\cite[Prop. 2.21(2)]{EL}.
\end{proof}

Therefore Conjecture~\ref{conj_augm} implies that any strong exceptional collection of line bundles of maximal length is full.

\begin{lemma}
\label{lemma_cseweak=mild}
Let $X$ be a rational surface and $A$ be a cyclic strong exceptional toric system on~$X$. Then $A$ is an augmentation in the weak sense if and only if $A$ is a cyclic strong exceptional augmentation. 
\end{lemma}
\begin{proof}
``If'' is trivial. For ``only if'', suppose  $A=t_k\circ\ldots\circ t_1(A')$ where $A'$ is a toric system on a Hirzebruch surface and any $t_i$ is either $\perm, \sh$ or $\augm$. Recall that $\perm$ and $\sh$ preserve cyclic strong exceptional toric systems (Lemma~\ref{lemma_perm}(3) and a remark after it). By Proposition~\ref{prop_augmexc}, if $\augm_k(B)$ is cyclic strong exceptional then so is $B$. Therefore all toric systems  $t_l\circ\ldots\circ t_1(A')$ and $A'$ where $1\le l\le k$ are cyclic strong exceptional and $A$ is a cyclic strong exceptional augmentation by definition. 
\end{proof}

Hille and Perling in their original paper \cite[Theorem 8.1]{HP} proved that any strong exceptional toric system on a toric surface is a strong exceptional augmentation. 
In \cite[Theorem 1.4]{EL} it is proved that any toric system on a del Pezzo surface is a standard augmentation. Hochenegger and Ilten in \cite{HI} prove that any exceptional toric system on a toric surface of Picard rank $\le 4$ is a standard augmentation. 

\begin{remark}
\label{remark_HI}
Hochenegger and Ilten  provide examples of a strong exceptional toric system $A=(A_1,\ldots,A_7)$ on a weak del Pezzo toric surface of Picard rank $5$ such that $A$ is not an elementary augmentation, see \cite[Example 5.6]{HI} and \cite[Section 4]{Ho}. We remark that the transposition $\perm_1(A)=(-A_1,A_1+A_2,A_3,\ldots,A_6,A_7+A_1)$ is also a strong exceptional toric system and  is an elementary (and moreover standard) augmentation. Therefore the cited examples do not contradict to  Conjecture~\ref{conj_augm}. 
\end{remark}

Here we give another  example demonstrating that the use of transpositions is necessary. 
Moreover, one cannot get a  ``normal form'' only by reordering of line bundles in the collection. Therefore the use of cyclic shifts (or of transpositions affecting the last term) is necessary.

\begin{example}
\label{example_bad}
Let $X=X_{4,2A_1,8}$ be the weak del Pezzo surface of degree $4$ of type $2A_1$ with $8$ lines. Explicitly, let $P_1,P_2,P_3,P_4\in\P^2$ be points and $H\subset \P^2$ be a line such that $P_1,P_2,P_3\in H$, $P_4\notin H$. Let $X'$ be the blow-up of $P_1,P_2,P_3,P_4$ with exceptional divisors $E_1,E_2,E_3,E_4$. Let $P_5\in E_4$ be a general point and $X$ be the blow-up of $P_5$.  Consider the following toric system on $X$ (recall that $E_{k_1,k_2,\ldots,k_m}$ denotes $E_{k_1}+E_{k_2}+\ldots+E_{k_m}$):
$$A=(A_1,\ldots,A_8)=(L-E_{145}, E_4, L-E_{234}, L-E_5, E_5-E_1, L-E_{35}, E_3-E_2, -L+E_{125}).
$$
We have 
$$A^2=(-2,-1,-2,0,-2,-1,-2,-2),$$
hence $A$ is of the first kind, type (8c) from Table~\ref{table_csadm}.

There are $8$ irreducible $(-1)$-classes on $X$, see \cite[Prop. 6.1(3)]{CT}: 
\begin{multline*}
E_1=A_{678}, E_2=A_{812}, E_3=A_{7812}, E_5=A_{5678}, L-E_{14}=A_{56781}, L-E_{24}=A_{78123}, \\
L-E_{34}=A_{8123}, L-E_{45}=A_{6781}.
\end{multline*}
It follows that no $A_i$ is a $(-1)$-curve, so $A$ is not an elementary augmentation.
Further, by Lemma~\ref{lemma_IF}(1) there are no irreducible $(-1)$-classes of the form $A_{k\ldots l}$ where $1\le k\le l\le 7$. This means (it would be clear in Section~\ref{section_firstkind}, see Lemma~\ref{lemma_IFdef}) that there exist no elementary augmentation $B$ that can be sent to $A$ by transpositions $\perm_1,\ldots,\perm_7$. On the other hand, let $B=\perm_7\perm_8(A)$, then $B_6=E_1$ is an irreducible curve and $B$ is an elementary augmentation.

One can check (using Theorem~\ref{theorem_checkonlyminustwo}) that the toric system $A$ is cyclic strong exceptional.
\end{example}

\section{Toric systems of the first kind are augmentations}
\label{section_firstkind}

Below we prove that any toric system $A=(A_1,\ldots,A_n)$ of the first kind (Definition~\ref{def_1kind}) on a rational surface $X$ is an augmentation in the weak sense (Definition~\ref{def_augm2}). Moreover, if $A$ is exceptional/strong  exceptional/cyclic strong exceptional then $A$ is an exceptional/strong  exceptional/cyclic strong exceptional augmentation in the sense of Definition~\ref{def_augm4}.

The proof is by induction in $n$. To do one step, it suffices to find a toric system $B$ obtained from $A$ by several transpositions, such that $B$ is an elementary augmentation. For any toric system $A$ on a surface $X$, denote by $I(X,A)\subset I(X)$ the subset of $(-1)$-classes which are elements of toric systems that can be obtained from $A$ by several transpositions:
$$I(X,A)=\{D\in I(X)\mid \exists k_1,\ldots,k_r,i\in [1,\ldots,n] \quad D=B_i,\,\,\text{where}\,\, B=\perm_{k_r}\ldots \perm_{k_1}(A)\}.$$
\begin{lemma}
\label{lemma_IFdef}
In the above notation one has
\begin{equation}
\label{eq_IF}
I(X,A)=\{A_{k\ldots l}\mid \exists m,\, k\le m\le l,\, a_k=\ldots=a_{m-1}=a_{m+1}=\ldots=a_l=-2,\,  a_m=-1\},
\end{equation}
where $a_i=A_i^2$. 
\end{lemma}
\begin{proof}
First, we prove $\supset$ inclusion. 
Suppose $a_k=\ldots=a_{m-1}=a_{m+1}=\ldots=a_l=-2,\,  a_m=-1$. 
Note that both sides of \eqref{eq_IF} are invariant under cyclic shifts of $A$ (for $I(X,A)$ it holds because $\perm$ commutes with $\sh$: $\perm_i\circ \sh=\sh\circ \perm_{i+1}$). Hence, applying cyclic shifts to~$A$, we may and will assume that $1\le k\le l\le n-1$.
Let 
\begin{equation}
\label{eq_oxd}
(\O_X(D_1),\ldots,\O_X(D_n))
\end{equation}
be the corresponding numerically exceptional collection. Then the line bundles in any of the segments $\O_X(D_k),\ldots,\O_X(D_m)$ and 
$\O_X(D_{m+1}),\ldots,\O_X(D_{l+1})$ are numerically completely orthogonal. Permuting line bundles in \eqref{eq_oxd} one can obtain the numerically exceptional collection
\begin{multline*}
(\O_X(D_1),\ldots,\O_X(D_{k-1}),\quad \O_X(D_{k+1}),\O_X(D_{k+2}),\ldots,\O_X(D_m),\O_X(D_k),\\
\O_X(D_{l+1}),\O_X(D_{m+1}),\ldots,\O_X(D_{l}),\quad \O_X(D_{l+2}),\ldots,\O_X(D_n)).
\end{multline*}
Let $B$ be the corresponding toric system, explicitly one has
$$B=\perm_{m+1}\ldots \perm_{l-1}\perm_l (\perm_{m-1}\ldots \perm_{k+1}\perm_k(A))$$
(recall that $\perm_i(A)^2=A^2$ and thus  one has $(\perm_{k_r}\ldots \perm_{k_1}(A))_i^2=-2$ for any $i,k_1,\ldots,k_r\in [k,\ldots,l]\setminus\{m\}$). 
Then $B_m=D_{l+1}-D_k=A_{k\ldots l}$.
Since $A_{k\ldots l}^2$ is a $(-1)$-class by Proposition~\ref{prop_Akl2}, we conclude that $A_{k\ldots l}\in I(X,A)$.

Now we prove $\subset$ inclusion. As above, applying cyclic shifts to $A$ we may assume that $a_n\ne -2$ (otherwise $A^2=(-2,-2,\ldots,-2)$, this sequence is not admissible).
Suppose that $A'_m$ is a $(-1)$-class for some toric system 
\begin{equation}
\label{eq_permperm}
A'=\perm_{k_r}\ldots \perm_{k_1}(A).
\end{equation} 
Let $(\O_X(D_1),\ldots,\O_X(D_n))$ and $(\O_X(D'_1),\ldots,\O_X(D'_n))$ be the corresponding collections of line bundles.
The sets $\{D_1,\ldots,D_n\}$ and $\{D'_1,\ldots,D'_n\}$ are the same because $a_n\ne -2$ and $\perm_n$ is not used in \eqref{eq_permperm}. Assume 
$D'_m=D_k$ and $D'_{m+1}=D_{l+1}$. Hence $A'_m=D'_{m+1}-D'_m=D_{l+1}-D_k=A_{k\ldots l}$. Further, the line bundles $D_k,D_{k+1},\ldots,D_{m}$ are in one block, and 
$D_{m+1},D_{m+2},\ldots,D_{l+1}$ are in another block. It follows that $a_k=a_{k+1}=\ldots=a_{m-1}=a_{m+1}=\ldots=a_{l}=-2$. Also $A_m^2=(A'_m)^2=-1$.
\end{proof}

The following lemma is very encouraging.
\begin{lemma}
\label{lemma_IF}
Let $A$ be a toric system of the first kind on a rational surface $X$.
\begin{enumerate}
\item Suppose $C=A_{k\ldots l}$ and $C'=A_{k'\ldots l'}$ are $(-1)$-classes and $C=C'$. Then $k=k'$ and $l=l'$.
\item Assume $\deg(X)\le 7$, then $I(X,A)=I(X)$.
\end{enumerate}
\end{lemma}
\begin{proof}
(1) 
Denote $a_i:=A_i^2$.
By Lemma~\ref{lemma_IFdef}, there exist numbers 
$p\in [k,\ldots, l]$, $p'\in [k',\ldots, l']$ such that 
\begin{equation}
\label{eq_apap}
a_p=a_{p'}=-1, \quad a_i=-2\quad\text{for}\quad i=k,\ldots, p-1,p+1,\ldots,l \quad\text{and}\quad i=k',\ldots, p'-1,p'+1,\ldots,l'.
\end{equation}
We consider several cases distinguished by the relation between $[k,\ldots,l]$ and $[k',\ldots,l']$. Essentially, there are three cases (note that $[k,\ldots,l]\cup [k',\ldots,l']\ne [1,\ldots,n]$ by \eqref{eq_apap} and explicit description of admissible sequences of the first kind in Proposition~\ref{prop_csadm}).

First, suppose that $[k,\ldots,l]\subset [k',\ldots,l']$. Then
$$C'=A_{k'\ldots k-1}+A_{k\ldots l}+A_{l+1\ldots l'}=:D_-+C+D_+$$
and 
$$C\cdot C'=C^2+D_-\cdot C+D_+\cdot C.$$ Note that 
$D_-\cdot C=1$ unless $k=k'$ and $D_-=0$, similarly for $D_+\cdot C$.
Since $C^2=C\cdot C'=-1$, it follows that $k=k'$ and $l=l'$, we are done.

Now assume that $k<k'\le l<l'$. Then 
$$C=A_{k\ldots k'-1}+A_{k'\ldots l}=:D_-+D\quad\text{and}\quad
C'=A_{k'\ldots l}+A_{l+1\ldots l'}=:D+D_+,$$
where $D_-,D_+$ are nonzero. We have 
$$-1=C\cdot C'=D^2+D\cdot (D_-+D_+)+D_-\cdot D_+=D^2+2+0=D^2+2.$$ 
Since $D^2=-1$ or $-2$ by Proposition~\ref{prop_Akl2} and \eqref{eq_apap}, we get a contradiction.

Finally, assume that $[k,\ldots,l]\cap [k',\ldots,l']=\emptyset$. Then $-1=C\cdot C'=A_{k\ldots l}\cdot A_{k'\ldots l'}=0$ or $1$, a contradiction.

(2) Note that the sequence $A^2$ is admissible and of the first kind, see Definition~\ref{def_1kind}. For any admissible sequence $A^2$ of the first kind (up to a symmetry and a shift, see Table~\ref{table_csadm}) we find the cardinality of $I(X,A)$ using Lemma~\ref{lemma_IFdef}.  
In Table~\ref{table_ruby} we present all divisors $A_{k\ldots l}$ satisfying numerical conditions in the r.h.s. of~\eqref{eq_IF}. By part (1), they are pairwise different. Also, their number equals to the  cardinality of $I(X)$, see Table~\ref{table_root}. It follows that $I(X)=I(X,A)$.

\begin{table}[h]
\caption{$I(X,A)$ for toric systems of the first kind}
\label{table_ruby}
\begin{center}
\begin{tabular}{|c|c|c|c|}
		 \hline
			type & $A^2$ & $I(X,A)$ & $|I(X)|$\\
		 \hline	
			5a & $(0,0,-1,-1,-1)$ & $A_3,A_4,A_5$ & $3$\\
		 \hline	
			5b & $(0,-2,-1,-1,1)$ & $A_{23},A_3,A_4$ & $3$\\
		 \hline	
			6a & $(-1,-1,-1,-1,-1,-1)$ & $A_1,A_2,A_3,A_4,A_5,A_6$ & $6$\\
		 \hline	
			6b & $(-1,-1,-2,-1,-1,0)$ & $A_1,A_2,A_{23},A_{34},A_4,A_5$ & $6$\\
		 \hline	
			6c & $(0,0,-2,-1,-2,-1)$ & $A_{34},A_{345},A_4,A_{45},A_{56},A_6$ & $6$\\
		 \hline	
			6d & $(0,-2,-2,-1,-2,1)$ & $A_{2345},A_{234}, A_{345},A_{34},A_{45},A_{4}$ & $6$\\
		 \hline	
			7a & $(-1,-1,-1,-2,-1,-2,-1)$ & $A_1,A_2,A_3,A_{34},A_{45},A_{456},A_5,A_{56},A_{67},A_7$ & $10$ \\
		 \hline	
			7b & $(-2,-1,-2,-2,-1,-1,0)$ & $A_{1234},A_{234},A_{123},A_{23},A_{12},A_{2}, A_{345},A_{45},A_5,A_6$ & $10$\\
		 \hline	
			8a & $(-2,-1,-2,-1,-2,-1,-2,-1)$ & $A_{2k},A_{2k-1,2k},A_{2k,2k+1}, A_{2k-1,2k,2k+1}$ & $16$\\
			&&  for $k\in \{1,2,3, 4\}$ & \\
		 \hline	
			8b & $(-2,-1,-1,-2,-1,-2,-2,-1)$ & $A_{12},A_2,A_3,A_{34}, A_{k\ldots l}$ for 
			$k\in\{4,5\}, l\in\{5,6,7\}$, & $16$\\
			& & $A_{k\ldots l}$ for $k\in \{6,7,8\}, l\in\{8,1\}$ & \\
		 \hline	
			8c & $(-2,-1,-2,-2,-2,-1,-2,0)$  & $A_{k\ldots l}$ for $k\in \{1,2\}, l\in\{2,3,4,5\},$ & $16$\\
			& & $A_{k\ldots l}$ for $k\in \{3,4,5,6\}, l\in\{6,7\}$ & \\
		 \hline	
			9 & $(-2,-2,-1,-2,-2,-1,-2,-2,-1)$ & $A_{k\ldots l}$ for $k\in\{1,2,3\},j\in\{3,4,5\}$,& $27$ \\
			&& $A_{k\ldots l}$ for $k\in\{4,5,6\},l\in\{6,7,8\}$, & \\
			&& $A_{k\ldots l}$ for $k\in\{7,8,9\},l\in\{9,1,2\}$ & \\
		 \hline	
\end{tabular}
\end{center}
\end{table}

\end{proof}

\begin{theorem}
\label{theorem_1kindweak}
Any toric system $A$ of the first kind on a smooth rational projective surface~$X$ is an augmentation in the weak sense. Moreover, if $X\ne \P^2$ then $A$ is an augmentation in the weak sense along any given chain of blow-downs to a Hirzebruch surface.
\end{theorem}
\begin{proof}
If $X$ is a Hirzebruch surface or $\P^2$ there is nothing to prove. Assume now there is a blow-up $X\to X'$ with the exceptional curve $E$. By Lemma~\ref{lemma_IF}, $I(X,A)=I(X)$. The class of $E$ belongs to $I(X,A)$, hence by Lemma~\ref{lemma_IFdef} $E=A_{k\ldots l}$ for some $k,l$. 
By definition of $I(X,A)$, there exists a toric system of the form $$B=\perm_{k_r}\ldots \perm_{k_1}(A),$$ 
such that  $B_m=E$ for some $m$. By Proposition~\ref{prop_elemaugm}, $B=\augm_m(C)$ for  some toric system $C$ on $X'$. We have $B^2=A^2$ by \eqref{eq_perm2}, hence $B$ is of the first kind. Now $C$ is also of the first kind by Lemma~\ref{lemma_bld1kind}. By induction in $\rank \Pic(X)$ we may assume that $C$ is an augmentation in the weak sense, hence $B$ and $A$ are also such.
\end{proof}


\begin{corollary}
\label{cor_cyclicaugm}
Any cyclic strong exceptional toric system $A$ on a rational surface~$X$ is a cyclic strong exceptional augmentation along any given chain of blow-downs of $X$ to a Hirzebruch surface. 
\end{corollary}
\begin{proof}
Any cyclic strong exceptional toric system is of the first kind by Proposition~\ref{prop_cyclic1type}. 
By Theorem~\ref{theorem_1kindweak}, $A$ is an augmentation in the weak sense along any chain of blow-downs to a Hirzebruch surface. From Lemma~\ref{lemma_cseweak=mild} it follows that $A$ is a cyclic strong exceptional augmentation along that chain. 
\end{proof}

\begin{corollary}
Any cyclic strong exceptional collection of line bundles of maximal length on a rational surface is full.
\end{corollary}
\begin{proof}
It follows from Corollary~\ref{cor_cyclicaugm} and Proposition~\ref{prop_full}.
\end{proof}

\begin{corollary}
\label{cor_cyclicblowup}
Let $X'$ be a rational surface and $X$ be a blow-up of a point on $X'$. Suppose that $X$ possesses  a cyclic strong exceptional toric system. Then $X'$ also possesses such system.
\end{corollary}
\begin{proof}
Let $A$ be a cyclic strong exceptional toric system on $X$. By Corollary~\ref{cor_cyclicaugm}, $A$ is a cyclic strong exceptional augmentation along the blow-down $X\to X'$. By Definition~\ref{def_augm4}, it means that there exists a cyclic strong exceptional toric system $B$ on $X$ (obtained from $A$ by transpositions and cyclic shifts) such that  $B=\augm_i(B')$ for some cyclic strong exceptional toric system $B'$ on $X'$. 
\end{proof}

\begin{corollary}
\label{corollary_f012}
A rational surface $X$ possessing a cyclic strong exceptional toric system cannot be blown down to $\mathbb F_d$ for $d>2$. Moreover,  $X\cong \mathbb F_0,\mathbb F_2$ or $X$ has $\P^2$ as a minimal model.
\end{corollary}
\begin{proof}
First, we  note that Hirzebruch surfaces $\mathbb F_d$ do not have cyclic strong exceptional toric systems for $d>2$. It follows from \cite[Proposition 5.2]{HP}. Indeed, let $F,S\in\Pic(\mathbb F_d)$ denote the standard basis such that $F$ is a fiber and $S$ is a $d$-curve, $d\ge 0$. Then by \cite[Proposition 5.2]{HP} all exceptional toric systems on $\mathbb F_d$
have the form (up to cyclic shifts) 
$$(F,S+sF,F,S-(d+s)F),$$
moreover, such system is cyclic strong exceptional if and only if $s,-(d+s)\ge -1$. 
The existence of a cyclic strong exceptional toric system implies that $-d=s-(d+s)\ge -2$, that is, $d\le 2$.

By the classification of minimal rational surfaces, $X$ can be blown down to $\P^2$ or to some $\mathbb F_d$. Assume the latter case. It follows from the above remark and 
 Corollary~\ref{cor_cyclicblowup} that $0\le d\le 2$. If $X$ is not $\mathbb F_d$ then $X$ can be blown down to $X'$ which is a blow up of $\mathbb F_0,\mathbb F_1$ or $\mathbb F_2$ at one point.
Note that such $X'$ can be always blown down to $\mathbb F_1$ or $\mathbb F_3$. The latter case is impossible by Corollary~\ref{cor_cyclicblowup} so $X'$ can be blown down to $\mathbb F_1$ and further to $\P^2$.
\end{proof}

By Theorem~\ref{theorem_1kindweak}, any exceptional toric system of the first kind can be obtained by elementary augmentations and transpositions from some toric system on a Hirzebruch surface. But Theorem~\ref{theorem_1kindweak} does not track exceptionality of toric systems. The following result guarantees that the intermediate toric systems can be chosen to be also exceptional (and thus we do not leave the world of exceptional collections of line bundles).

\begin{theorem}
\label{theorem_1kind}
\begin{enumerate}
\item 
Any exceptional toric system $A$ of the first kind on a rational surface $X$ is an exceptional augmentation. 
\item
Any strong exceptional toric system $A$ of the first kind on a weak del Pezzo surface $X$ is a strong exceptional augmentation. 
\end{enumerate}
\end{theorem}
\begin{proof}
(1) In fact, all intermediate toric systems from the proof of Theorem~\ref{theorem_1kindweak} are also exceptional.  
To make the arguments more clear, we give an independent proof. 

We prove that any exceptional toric system $A$ of the first kind on a rational surface $X$ is an exceptional augmentation by  induction in $\rank\Pic(X)$. Cases $X=\P^2$ or $X$ is a Hirzebruch surface are trivial.
Assume $X$ is the blow-up of $X'$ and $E$ is
the $(-1)$-curve. Note that  we have $\deg(X)\ge 3$ by Proposition~\ref{prop_csadm}. By Lemma~\ref{lemma_IF} we have $E = A_{k\ldots l}$ for some cyclic segment $[k,\ldots,l]$.
For any fixed $\rank\Pic(X)$ we argue by induction in the length of that cyclic segment. For $l=k$ we have that $A_k=E$ is a $(-1)$-curve, $A=\augm_k(C)$ by Proposition~\ref{prop_elemaugm} where the toric system $C$ on $X'$ is exceptional by Proposition~\ref{prop_augmexc}
and we may use induction in $\rank\Pic(X)$. For $l\ne k$ by 
Lemma~\ref{lemma_IFdef} we have 
\begin{equation}
\label{eq_klm}
A_m^2=-1\quad\text{for some $m$, $k\le m\le l$},\quad 
A_i^2=-2\quad\text{for}\quad i=k,k+1,\ldots,m-1,m+1,\ldots,l.
\end{equation}
We may assume without loss of generality that  $m\ne l$. The $(-2)$-class $A_l$ is not effective. Indeed, otherwise $E=A_{k\ldots l}=A_{k\ldots l-1}+A_l$ and thus is reducible, because $A_{k\ldots l-1}$ is a $(-1)$-class by Proposition~\ref{prop_Akl2} and hence is effective by Lemma~\ref{lemma_sqrt}(2). 
This gives a contradiction since a $(-1)$-curve is not linearly equivalent to any other effective divisor and by our assumption $A_{k\ldots l}$ is a class of an irreducible curve.

By Lemma~\ref{lemma_effslo}(2) we have $A_l\in R^{\rm lo}\setminus R^{\rm eff}=R^{\rm slo}$. It follows (see Lemma~\ref{lemma_perm}(1)) that $A'=\perm_l(A)$ is also an exceptional toric system of the first kind on $X$ and 
$(A')_{k\ldots l-1}=A_{k\ldots l}=E$. By the induction hypothesis, $A'$ is an exceptional augmentation.

(2) The proof is analogous to (1). We prove that any strong exceptional toric system $A=(A_1,\ldots,A_n)$ of the first kind on a weak del Pezzo surface $X$ is a strong exceptional augmentation. We argue by induction in $\rank\Pic (X)$. For the induction step, assume $X\to X'$ is a blow-up with an exceptional curve $E$. As in (1), we have $E=A_{k\ldots l}$ for some cyclic segment $[k,\ldots,l]\subset [1,\ldots,n]$. For any fixed $\rank\Pic (X)$ we use induction in the length of $[k,\ldots,l]$. As in (1), for $k=l$ we have that $A=\augm_k(C)$ where $C$ is a strong exceptional toric system on $X'$ by Propositions~\ref{prop_elemaugm} and~\ref{prop_augmexc}. One can use induction in 
$\rank\Pic (X)$.

For $k\ne l$ we can assume as in (1) that \eqref{eq_klm} holds, $m\ne l$ and $A_l$ is a strong left-orthogonal $(-2)$-class. In contrast with the proof of (1), now the toric system $\perm_l(A)$ may be not strong exceptional. If $l\ne n$ then $\perm_l(A)$ is strong exceptional by Lemma~\ref{lemma_perm}(2) and there are no problems. Now suppose $l=n$. We claim that the system $\sh(A)=(A_2,\ldots,A_n,A_1)$ is  strong exceptional (recall that $\sh(A)$ is automatically exceptional). Assume the contrary. Then by Proposition~\ref{prop_straightforward}(1),(2) there exists a divisor $D=A_{p\ldots q}$ with $2\le p\le q\le n$ which is left-orthogonal but not strong left-orthogonal (note that Proposition~\ref{prop_straightforward} is applied to $\sh(A)$, not $A$). If $q<n$ then $A_{p\ldots q}$ is strong left-orthogonal since $A$ is strong exceptional. Hence, $q=n=l$. If $p\le m$ then $D$ is a left-orthogonal $r$-class with 
$$r=D^2=-2+\sum_{i=p}^n(A_i^2+2)\ge -2+A_m^2=-1$$ by Proposition~\ref{prop_Akl2}. Therefore, $D$ is strong left-orthogonal by Proposition~\ref{prop_loslo}. Hence, $m+1\le p$. It means that $D$ is a $(-2)$-class (by Proposition~\ref{prop_Akl2}) which is left-orthogonal but not strong left-orthogonal. By Lemma~\ref{lemma_effslo}(2), $D$ is effective. Both $D$ and $A_{k\ldots p-1}$ (as a $(-1)$-class) are effective, therefore $A_{k\ldots l}=A_{k\ldots p-1}+D$ is reducible, a contradiction. Thus the toric system
$\sh(A)=(A_2,\ldots,A_n,A_1)$ is  strong exceptional. Now the toric system $$A''=\perm_{n-1}\sh(A)=(A_2,\ldots,A_{n-2},A_{n-1,n},-A_n,A_{1,n})$$ 
is  also strong exceptional (by Lemma~\ref{lemma_perm}(2)) and $(A'')_{k-1,\ldots,n-2}=A_{k\ldots n}=E$ is irreducible. We proceed by induction in $l-k$. 
\end{proof}

\section{Surfaces with cyclic strong exceptional toric systems}
\label{section_cyclicstrong}

In this section we classify 
surfaces with cyclic strong exceptional collections of line bundles having maximal length. We prove that such collections  can exist only on weak del Pezzo surfaces and  determine which types of weak del Pezzo surfaces possess such collections and which do not.

First we demonstrate that a surface with a cyclic strong exceptional collection of line bundles having maximal length must be rational. It is proven by Morgan Brown and Ian Shipman in \cite[Theorem 4.4]{BS} that a surface with a full strong exceptional collection of line bundles is rational. It seems that fullness is not really needed in the proof in \cite{BS}, nevertheless, we prefer to give a separate simple proof for the case of cyclic strong exceptional  collections.

\begin{lemma}
\label{lemma_rational}
Let $X$ be a smooth projective surface such that the sheaf $\O_X$ is exceptional. Assume $X$ admits a cyclic strong exceptional toric system. Then $X$ is a rational surface.
\end{lemma}

\begin{proof}
Assume that $(A_1,\ldots,A_n)$ is a cyclic strong exceptional toric system on $X$.  By Proposition~\ref{prop_A2adm}, the sequence $A^2$ is admissible. 
By Proposition~\ref{prop_cyclic1type}, we have  $A_i^2\ge -2$ for all $i$. Therefore, $A^2$ is of the first kind, see Table~\ref{table_csadm}. Then we can group $A_1+\ldots+A_n$ into two groups $D_1:=A_1+\ldots+A_m, D_2:=A_{m+1}+\ldots+A_n$ for some $m$ such that $D_1,D_2$ are both non-zero strong left-orthogonal divisors and $D_1^2\ge -1, D_2^2\ge -1$. Thus $h^0(D_i)=D_i^2+2>0$ for $i=1,2$. Then $D_1,D_2$ are both effective, hence $-K_X=A_1+\ldots+A_n$ is effective and $-2K_X$ is effective and nonzero. It follows immediately that $h^0(2K_X)=0$. Recall that $h^1(\O_X)=0$ by the assumptions. By Castelnuovo's rationality criterion, we conclude that $X$ is a rational surface.
\end{proof}

\begin{predl} 
\label{prop_cyclicwdp}
Let $X$ be a smooth projective surface such that the sheaf $\O_X$ is exceptional. Assume $X$ admits a cyclic strong exceptional toric system. Then~$X$ is a weak del Pezzo surface of degree $\deg(X)\ge 3$. 
\end{predl}
\begin{proof}
By Lemma~\ref{lemma_rational}, $X$ must be a rational surface. By Corollary~\ref{corollary_f012}, $X$ is either $\mathbb F_0, \mathbb F_2$, or a blow-up of $\P^2$.
Since $\mathbb F_0, \mathbb F_2$ are weak del Pezzo surfaces, we have to consider $X$ being a blow-up of $\P^2$. By \cite[Prop. 8.1.23]{Do} or \cite[Prop. 0.4]{CT}, it suffices to show that this blow-up is in \emph{almost general position}. It means that there is a sequence of $s\le 8$ blow-ups
$$X=X_s\to X_{s-1}\to\ldots \to X_1\to X_0=\P^2$$
where $X_m\to X_{m-1}$ is a blow-up of one point $P_m\in X_{m-1}$ which does not lie on a smooth rational $(-2)$-curve on $X_{m-1}$.

Note that in our case $s\le 6$ and $\deg(X)\ge 3$ because $X$ possesses a cyclic strong exceptional toric system, see Proposition~\ref{prop_csadm}. Denote by $L,E_1,\ldots,E_m$ the standard basis in $\Pic(X_m)$. Then for $0\le m\le 5$ all $(-2)$-classes in $\Pic(X_m)$ are of the form $E_i-E_j$ and $\pm(L-E_i-E_j-E_k)$ (see Appendix A or \cite[a table in 25.5.2]{Ma}). Suppose that for some $0\le m\le 5$ the point $P_m$ belongs to a $(-2)$-curve $R\subset X_{m-1}$ which is of the form $E_i-E_j$ or $L-E_i-E_j-E_k$, where $i<j<k<m$.

In the first case, consider the divisor 
$$D=L-E_j-E_m=(L-E_i)+(E_i-E_j-E_m).$$
This divisor is a $(-1)$-class, so $D\in I(X,A)$ by Lemma~\ref{lemma_IF} (where $A$ is a cyclic strong exceptional toric system on $X$).
It follows from Proposition~\ref{prop_straightforward} and Lemma~\ref{lemma_IFdef} that $D$ is strong left-orthogonal. On the other hand, the class $E_i-E_j-E_m$ is effective
and we have (using Proposition~\ref{prop_D2})
$$h^0(D)\ge h^0(L-E_i)=2>D^2+2=\chi(D),$$
thus $D$ cannot be strong left-orthogonal. 

In the second case 
the class $L-E_i-E_j-E_k-E_m\in\Pic(X_m)$ is  represented by a $(-3)$-curve $C\subset X_m$ (because the class $L-E_i-E_j-E_k\in\Pic(X_{m-1})$ is represented by a $(-2)$-curve containing the point $P_m$). Let $f_l\colon X_l\to\P^2$ denote the blow-down. Then $f_m(C)$ is a line on $\P^2$ and
$$L=f_m^*L=f_m^*f_{m}(C)=C+\sum_{l=1}^m c_lE_l,$$
where $c_l>0$ if and only if $f_l(P_l)\in f_m(C)$, otherwise $c_l=0$.
On the other hand, $L=C+E_i+E_j+E_k+E_m$.
It follows that for any  $l\le m, l\notin\{i,j,k,m\}$ the point $f_l(P_l)$ does not belong to $f_m(C)$, while for $l\in\{i,j,k,m\}$ we have $f_l(P_l)\in f_m(C)$. Hence the blow-ups of  $P_i,P_j,P_k,P_m$ commute with the blow-ups of other $P_l$-s. Changing the order of blow-ups we can assume that $(i,j,k,m)=(1,2,3,4)$ and other blow-ups are performed after these four. By Corollary~\ref{cor_cyclicblowup},
$X_4$ also has a cyclic strong exceptional toric system, denote it by  $(A_1,\ldots,A_7)$. We claim that there exists a cyclic segment $[p,\ldots, q]\subset [1,\ldots, 7]$ such that  $A_{p\ldots q}=2L-E_1-E_2-E_3-E_4$. It would follow from Proposition~\ref{prop_straightforward} that the divisor $D=2L-E_1-E_2-E_3-E_4$ is strong left-orthogonal. 
Since the divisor $D-L=L-E_1-E_2-E_3-E_4$ is effective, we have 
$$h^0(X_4,D)\ge h^0(X_4,L)=3>D^2+2=\chi(X_4,D),$$
and $D$ cannot be strong left-orthogonal.

Now we prove the claim. Note that $D$ is a $0$-class in $\Pic(X_4)$. 
Actually we will check that all $0$-classes in  $\Pic(X_4)$ are of the form $A_{k\ldots l}$. By writing $D=aL+\sum_{i=1}^4b_iE_i$ and solving equations $D^2=0, D\cdot K_X=-2$, one can check that there are totally  five $0$-classes: $2L-E_1-E_2-E_3-E_4$ and $L-E_i$, $i=1,2,3,4$.
By Proposition~\ref{prop_csadm}, the sequence $(A_1^2,\ldots,A_7^2)$ can be either 
$$(-1,-1,-1,-2,-1,-2,-1)\quad\text{or}\quad (-2,-1,-2,-2,-1,-1,0)$$ 
(up to a cyclic shift and a symmetry). 
In both cases one can check that the following 5-tuples of divisors are pairwise different $0$-classes:
$$(B_1,\ldots,B_5)=(A_{12},A_{23},A_{234}, A_{345}, A_{3456}),\qquad (C_1,\ldots,C_5)=(A_{56}, A_{456}, A_{3456}, A_7, A_{71})$$ 
respectively. Indeed, $B_i$ and $C_i$ are $0$-classes by Proposition~\ref{prop_Akl2}. 
In the first case the matrix $(A_i\cdot B_j)_{i\in\{1,2,5\}, 1\le j\le 5}$ has pairwise different columns, therefore $B_1,\ldots,B_5$ are also different. In the second case we use the matrix $(A_i\cdot C_j)_{i\in\{2,5,6\}, 1\le j\le 5}$. 
The matrices can be easily computed using the definitions of a toric system and of $A_{k\ldots l}$:
$$(A_i\cdot B_j)_{i\in\{1,2,5\}, 1\le j\le 5}=
\begin{pmatrix}
0&1&1&0&0\\
0&0&0&1&1\\
0&0&1&0&1
\end{pmatrix},\qquad
(A_i\cdot C_j)_{i\in\{2,5,6\}, 1\le j\le 5}=
\begin{pmatrix}
0&0&1&0&1\\
0&1&1&0&0\\
0&0&0&1&1
\end{pmatrix}.$$
\end{proof}

\medskip
Now we classify surfaces possessing a cyclic strong exceptional toric system.
We refer to Appendix~A for the sets of $(-2)$-curves on all types of weak del Pezzo surfaces of degree $\ge 3$. We provide explicit examples of cyclic strong exceptional toric systems on any surface where such a toric system exists. 
Recall notation \eqref{eq_Lij}:
\begin{equation*}
E_{i_1\ldots i_k}=E_{i_1}+\ldots+E_{i_k},\quad L_{i_1\ldots i_k}=L-E_{i_1\ldots i_k}.
\end{equation*}

\begin{predl}
\label{prop_classification}
Toric systems from Table~\ref{table_yes} are cyclic strong exceptional. Therefore weak del Pezzo surfaces from Table~\ref{table_yes} admit cyclic strong exceptional toric systems. Weak del Pezzo surfaces~$X$ from Table~\ref{table_no} do not admit cyclic strong exceptional toric systems. 
\end{predl}
\begin{table}[h]
\begin{center}
\caption{Cyclic strong exceptional toric systems}
\label{table_yes}
\begin{tabular}{|p{1cm}|p{6.7cm}|p{9cm}|}
 \hline
 degree & types & toric system \\
 \hline
 $9$ & $\P^2$ & $L,L,L$  \\
 \hline
 $8$ & $\mathbb F_0$ & $H_1,H_2,H_1,H_2$  \\
 \hline
 $8$ & $\mathbb F_1$ & $L_1,E_1,L_{1},L$  \\
 \hline
 $8$ & $\mathbb F_2$ & $F,S-F,F,S-F$ (where $F^2=0, S^2=2, FS=1$)  \\
 \hline
 $7$ & any & $L_{1},E_1,L_{12},E_2,L_{2}$  \\
 \hline
 $6$ & any & $L_{13},E_1,L_{12},E_2,L_{23},E_3$  \\
 \hline
 $5$ &  $(\emptyset);\,\, (A_1);\,\, (2A_1);\quad (A_2);\quad (A_1{+}A_2)$ &    ${L_{134}},E_4,{E_1{-}E_4},L_{12},E_2,L_{23},E_3$\\
 \hline
 $4$ & $(\emptyset); \,\,(A_1);\,\, (2A_1,9);\,\, (2A_1,8);\quad (A_2);$ &  ${L_{134}},E_4,{E_1{-}E_4},L_{12},{E_2{-}E_5},E_5,{L_{235}},E_3$  \\
     & $(3A_1);\quad (A_1{+}A_2); \quad (A_3,4); \quad (4A_1);$ & \\
     & $(2A_1{+}A_2);\quad (A_1{+}A_3);\quad (2A_1{+}A_3)$ & \\  
     \hline
 $3$ & $(\emptyset); \quad (A_1);\quad (2A_1);\quad   (A_2) ; \quad (3A_1) ;$ & $E_2{-}E_4,L_{125},E_5,E_1{-}E_5,L_{136},E_6,E_3{-}E_6,L_{234},E_4$   \\  
 & $(A_1{+}A_2); \quad  (4A_1);\quad (2A_1{+}A_2);$ & \\ 
   & $(2A_2);\quad (A_1{+}2A_2);\quad (3A_2)$ & \\
   \hline      
\end{tabular}
\end{center}
\end{table}

\begin{table}[h]
\begin{center}
\caption{Surfaces  with no cyclic strong exceptional toric systems}
\label{table_no}
\begin{tabular}{|p{2cm}|p{4cm}||p{3cm}|p{3cm}|}
 \hline
 $\deg(X)$ & type of $X$ & type of $X'$ & $P\in X'$ \\
 \hline
 $5$ & $A_3$ & & \\  \hline
 $5$ & $A_4$ & &\\  \hline
 $4$ & $A_3,5$ & $A_3$ & general \\  \hline
 $4$ & $A_4$ & $A_4$ & general \\  \hline
 $4$ & $D_4$ & $A_3$ & general on $L_{12}$\\  \hline
 $4$ & $D_5$ & $A_4$ & general on $E_4$ \\  \hline
 $3$ & $A_3$ & $A_3,5$ & general \\  \hline
 $3$ & $A_1+A_3$ & $A_3,5$ & general on $E_1$ \\  \hline

 $3$ & $A_4$ & $A_4$ & general \\  \hline
 $3$ & $D_4$ & $D_4$ & general \\ \hline
 $3$ & $2A_1+A_3$ & $A_3,5$ & $E_1\cap Q$  \\  \hline
 $3$ & $A_1+A_4$ & $A_4$ & general on $Q$ \\  \hline
 $3$ & $A_5$ & $A_4$ & general on $E_5$ \\  \hline
 $3$ & $D_5$ & $D_5$ & general \\  \hline
 $3$ & $A_1+A_5$ & $A_4$ & $E_5\cap Q$ \\  \hline
 $3$ & $E_6$ & $D_5$ & general on $E_5$\\  \hline
\end{tabular}
\end{center}
\end{table}

\begin{proof}
First note that the sequences from Table~\ref{table_yes} are toric systems. Moreover, it is easy to see that they are obtained from each other by  elementary augmentations; thus any sequence  in Table~\ref{table_yes} is a standard augmentation. Next one needs to check that they are cyclic strong exceptional. Any toric system $A$ in Table~\ref{table_yes} is of the first kind. According to Theorem~\ref{theorem_checkonlyminustwo}, we have to check that all divisors of the form $A_{k\ldots l}$ where $[k,\ldots, l]\subset [1,\ldots, n]$ is a cyclic segment such that $A_k^2=A_{k+1}^2=\ldots=A_l^2=-2$ are neither effective nor anti-effective. 
We see that for $\deg(X)\ge 6$ there are no such divisors. For  $\deg(X)=5,4,3$ we have the following divisors:
\begin{align*}
\deg(X)=5\colon & L_{134},E_1-E_4; \\
\deg(X)=4\colon & L_{134},E_1-E_4, E_2-E_5, L_{235};\\
\deg(X)=3\colon & E_2-E_4,L_{125},L_{145},E_1-E_5,L_{136},L_{356},E_3-E_6,L_{234},L_{246}.
\end{align*}
We refer to Tables~\ref{table_d5}, \ref{table_d4}, \ref{table_d3} for the sets $R^{\rm irr}(X)$ of $(-2)$-curves on weak del Pezzo surfaces of degrees $5,4,3$ (see notation in Section~\ref{subsection_rclasses}). Using these Tables and Proposition~\ref{prop_roots}, one has to check that the divisors listed above are not in $\pm R^{\rm eff}$.
It may be useful to note that $R^{\rm eff}(X_{5,A_1+A_2})$ is maximal among all surfaces of degree $5$ from Table~\ref{table_yes}, therefore we need to check the above claim  only for $X_{5,A_1+A_2}$. Likewise, for degree $4$ the maximal $R^{\rm eff}$ is owned by the surfaces $X_{4,2A_1+A_3}$ and $X_{4,2A_1+A_2}$. For degree $3$ the surfaces with maximal $R^{\rm eff}$ in Table~\ref{table_yes} are $X_{3,3A_2}$ and $X_{3,4A_1}$.

Now we prove the negative part of the Proposition.
First, we check that the surfaces $X_{5,A_3}$ and $X_{5,A_4}$ have no cyclic strong exceptional toric system. Assume the contrary, $A$ is such system. Then by Propositions~\ref{prop_cyclic1type} and \ref{prop_csadm}, $A^2$ is $(-1,-1,-1,-2,-1,-2,-1)$ or $(-2,-1,-2,-2,-1,-1,0)$ (up to a shift and a symmetry). In both cases we see that there exist two $(-2)$-classes $D_1=A_i$ and $D_2=A_{i+2}$ which are strong left-orthogonal and such that $D_1\cdot D_2=0$. On the other hand, we have $R^{\rm slo}(X_{5,A_3})=\pm \{L_{123},L_{124},L_{134},L_{234}\}$ and $R^{\rm slo}(X_{5,A_4})=\emptyset$ due to Lemma~\ref{lemma_effslo}(1) and Table~\ref{table_d5}. It is easy to see that such $D_1$ and $D_2$ cannot exist.

We claim that all other surfaces $X$ from Table~\ref{table_no} are blow-ups of  $X_{5,A_3}$ or $X_{5,A_4}$. It would follow then from Corollary~\ref{cor_cyclicblowup} that they admit no cyclic strong exceptional toric systems. In the right two columns of Table~\ref{table_no} we present a surface $X'$ and a point $P\in X'$ such that the blow-up of $X'$ at $P$  has the same type as $X$ (here $Q$ denotes the $(-1)$-class $2L-E_{12345}\in\Pic(X_4)$). We refer to \cite{CT}, Propositions 6.1 and 8.4, for the verifying.

Indeed, let $p\colon X\to X'$ be the blow-up of a point $P\in X'$ with the exceptional divisor $E$. Then one has
$$R^{\rm irr}(X)=p^*(R^{\rm irr}(X'))\bigsqcup \{p^*(C)-E\mid C\in I^{\rm irr}(X'),\,\, C\ni P \}.$$
The r.h.s. of the above formula can be read from the diagrams in 
\cite[Propositions 6.1 and 8.4]{CT}.

\end{proof}

\section{An application to dimension of $\D^b(\coh X)$}
\label{section_application}
The paper \cite{BF} by Matthew Ballard and David Favero initiated the study of the  relation between the dimension of triangulated categories in the sense of Rafael Rouquier \cite{Rou} and full cyclic strong exceptional collections. 
Recall the definition of dimension. For an object 
$G$ of a triangulated category~$\TT$ define full subcategories $\langle G\rangle_i$, $i\ge 0$ as follows. Say that $F\in \langle G\rangle_0$ if $F$ is a  direct summand  in some finite direct sum of shifts of $G$. For $i>0$ say that $F\in\langle G\rangle_i$ if there exists an exact triangle $F_0\to F'\to F_{i-1}\to F_0[1]$ with
$F_0\in\langle G\rangle_0$, $F_{i-1}\in\langle G\rangle_{i-1}$, and $F$ is a direct summand in $F'$. Denote $\langle G\rangle=\cup_i \langle G\rangle_i$. Object $G$
is called a \emph{classical generator} if $\TT=\langle G\rangle$. Generator $G$ is said to have \emph{generation time} $n$ if $n$ is the minimal number such that $\TT=\langle G\rangle_n$. Rouquier defined dimension $\dim \TT$ of a triangulated category $\TT$ as the minimal possible generation time for all  classical generators in~$\TT$. 
\begin{conjecture}[Orlov]
For a smooth projective variety $X$ one has 
 $$\dim \D^b(\coh(X))=\dim X.$$
\end{conjecture}
The bound $\D^b(\coh(X))\ge \dim X$ was proven by Rouquier in \cite{Rou}, while the equality is known only for some special varieties. A useful tool for proving the conjecture was found by Ballard and Favero in \cite{BF}:

\begin{theorem}[See {\cite[Theorem 3.4]{BF}}]
\label{theorem_BF}
Let $X$ be a smooth proper variety over a perfect field $\k$, let $T\in \D^b(\coh(X))$ be a classical generator such that $\Hom^i(T,T)=0$ for $i\ne 0$. Denote 
$$i_0=\max\{i\mid \Hom^i(T,T\otimes\O_X(-K_X))\ne 0\}.$$
Then the generation time of $T$ is $\dim X+i_0$.
\end{theorem}

\begin{corollary}
Let $X$ be a smooth proper variety over a perfect field $\k$, let $(\EE_1, \ldots,\EE_n)$ be a full strong exceptional collection on $X$. Suppose
the generation time of the generator $T=\oplus_i \EE_i$ is equal to $\dim X$. Then the collection $(\EE_1, \ldots,\EE_n)$ is cyclic strong exceptional.
\end{corollary}
 
We are able to prove the converse statement for collections of line bundles on surfaces. 

\begin{predl}
\label{prop_timetwo}
Let $X$ be a smooth projective surface with a full cyclic strong exceptional collection 
$$(\O_X(D_1),\ldots,\O_X(D_n))$$ 
of line bundles. Then the generator $T=\oplus_{i=1}^n \O_X(D_i)$ of the category $\D^b(\coh (X))$ has generation time two. Therefore, dimension of $\D^b(\coh(X))$ is equal to $2$.
\end{predl}
\begin{proof}
By Proposition~\ref{prop_cyclicwdp}, $X$ is a weak del Pezzo surface and $d:=\deg(X)\ge 3$. By Theorem~\ref{theorem_BF}, we have to check that $H^i(X,\O_X(D-K_X))=0$ for any $i>0$ and any $D$ of the form $D_k-D_l$. There are three possible cases:
$$
D=\begin{cases}
D_k-D_l, k<l,\\
0,\\
D_k-D_l, k>l.
\end{cases}$$
For $D=D_k-D_l$ where $k<l$ we have 
$$H^i(X,\O_X(D_k-D_l-K_X))=\Hom^i(\O_X(D_l),\O_X(D_k-K_X))=0$$
for any $i>0$ because the collection
\begin{multline*}
(\O_X(D_l),\O_X(D_{l+1}),\ldots,\O_X(D_n),\\ \O_X(D_1-K_X),\O_X(D_{2}-K_X),\ldots,\O_X(D_k-K_X),\ldots, \O_X(D_{l-1}-K_X))
\end{multline*}
is strong exceptional by assumption.

For $D=D_k-D_l$ where $k>l$ we borrow the arguments from Lemma 3.9 of \cite{BF}. By~\cite[Theorem 8.3.2]{Do}, the linear system $|-K_X|$ has no base points. By Bertini's Theorem, the general divisor $Z\in|-K_X|$ is a smooth irreducible curve. By the adjunction formula, $Z$ is a curve of genus~$1$.
Consider the standard exact sequence
\begin{equation}
\label{eq_DD}
0\to \O_X(D)\to \O_X(D-K_X)\to \O_Z(D-K_X)\to 0.
\end{equation}
One has $H^i(X,\O_X(D))=0$ for $i>0$ by strong exceptionality of the collection 
$(\O_X(D_1),\ldots,\O_X(D_n))$. Also, 
$$Z\cdot (D-K_X)=K_X(K_X-D)=d-D\cdot K_X=d+\chi(D)\ge d\ge 3$$
where the third equality is by Proposition~\ref{prop_D2} and the first inequality is because $D$ is a strong left-orthogonal divisor and hence $\chi(D)\ge 0$.
It follows that 
$$H^i(X,\O_Z(D-K_X))=H^i(Z,\LL)$$
where $\LL$ is a line bundle on $Z$ of degree $\ge 3$. Since $Z$ has genus $1$, we get $H^i(Z,\LL)=0$ for $i>0$.
The long exact sequence of cohomology associated to (\ref{eq_DD}) provides that $H^i(X,\O_X(D-K_X))=0$ for $i>0$.

The above arguments work also for $D=0$. Alternative, one can use Kawamata-Viehweg vanishing: $H^i(X,\O_X(-K_X))\cong H^{2-i}(X,\O_X(2K_X))^*=0$ for $i>0$ because $-2K_X$ is nef and big, see \cite[Theorem 2.64]{KM}.
\end{proof}

\begin{corollary}
Dimension of $\D^b(\coh(X))$ is equal to $2$ for any weak del Pezzo surface~$X$ from Table~\ref{table_yes}.
\end{corollary}

\section{A toric system which is not an augmentation}
\label{section_ce}

Here we give an example of a strong exceptional toric system $A$ on a weak del Pezzo surface~$X$ of degree $2$ such that $A$ is not an augmentation in the weak sense. 

By \cite[Section 3.6]{De}, a weak del Pezzo surface $X$ of degree $2$ of type $A_1+2A_3$ has the following configuration of $(-1)$-curves (denoted $\circ$) and $(-2)$-curves (denoted $\bul$). 
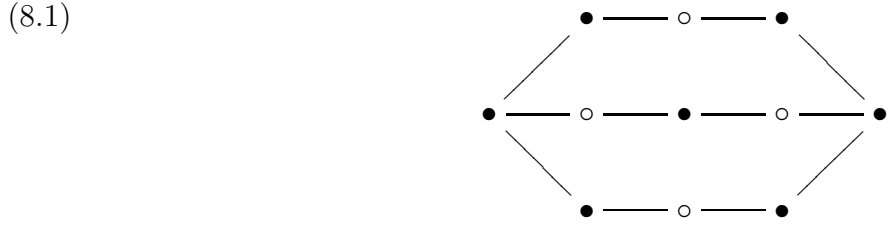
\begin{equation}
\label{eq_x1}
\xymatrix{& \bul \ar@{-}[r]\ar@{-}[ld] & \circ \ar@{-}[r] & \bul \ar@{-}[rd] &\\
\bul \ar@{-}[r]\ar@{-}[rd] & \circ \ar@{-}[r] & \bul\ar@{-}[r] & \circ \ar@{-}[r] & \bul\\
& \bul \ar@{-}[r] & \circ \ar@{-}[r] & \bul\ar@{-}[ru] & }
\end{equation}
Two vertices here are connected  by a line if and only if the corresponding curves intersect (then such intersection is unique and transversal).
There are four $(-1)$-curves on $X$, and they do not intersect. Let $p\colon X\to X'$ be the blow-down of these curves. One can check that $p$ maps $(-2)$-curves on $X$ to the following configuration of curves on $X'$ (where $*$ denotes a $0$-curve), and there are no other negative curves on $X'$:
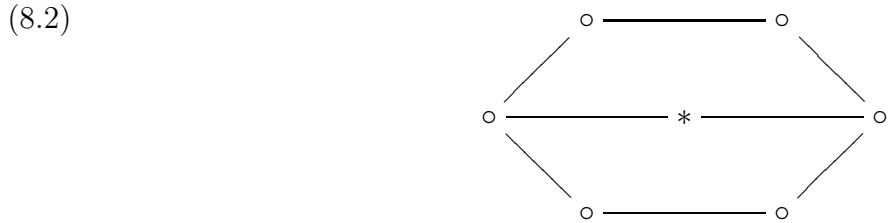
\begin{equation}
\label{eq_xprime1}
\xymatrix{& \circ \ar@{-}[rr]\ar@{-}[ld] &  & \circ \ar@{-}[rd] &\\
\circ \ar@{-}[rr]\ar@{-}[rd] &  & {*}\ar@{-}[rr] &  & \circ\\
& \circ \ar@{-}[rr] &  & \circ\ar@{-}[ru] & }
\end{equation}
It means that $X'$ is a blow-up of $\P^2$ at three generic points, and we may assume that the curves on \eqref{eq_xprime1} belong to the following classes (here we use notation \eqref{eq_Lij} from Section~\ref{subsection_rclasses}):
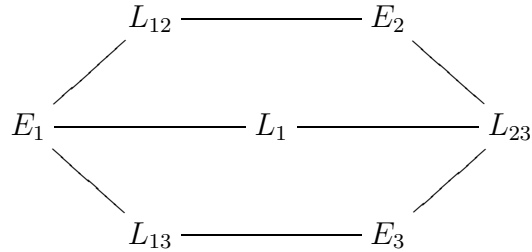
\begin{equation*}
\xymatrix{& L_{12} \ar@{-}[rr]\ar@{-}[ld] &  & E_2 \ar@{-}[rd] &\\
E_1 \ar@{-}[rr]\ar@{-}[rd] &  & L_1\ar@{-}[rr] &  & L_{23}\\
& L_{13} \ar@{-}[rr] &  & E_3\ar@{-}[ru] & }
\end{equation*}
Now $p$ is the blow-up of the four points 
$$P_4=E_2\cap L_{12}, P_5=E_3\cap L_{13}, P_6=E_1\cap C, P_7=L_{23}\cap C$$
on $X'$, where $C$ is some $0$-curve representing the class $L_1$. The curves on \eqref{eq_x1} belong to the following classes:
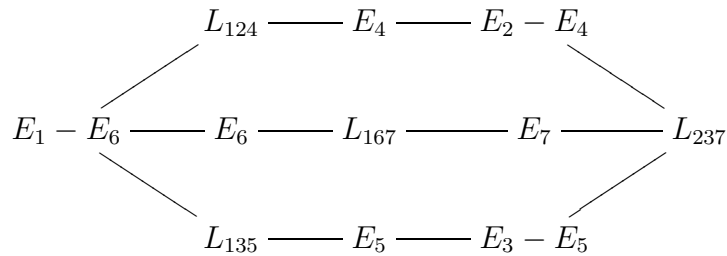
\begin{equation*}
\xymatrix{& L_{124} \ar@{-}[r]\ar@{-}[ld] & E_4 \ar@{-}[r] & E_2-E_4 \ar@{-}[rd] &\\
E_1-E_6 \ar@{-}[r]\ar@{-}[rd] & E_6 \ar@{-}[r] & L_{167} \ar@{-}[r] & E_7 \ar@{-}[r] & L_{237}\\
& L_{135}  \ar@{-}[r] & E_5 \ar@{-}[r] & E_3-E_5 \ar@{-}[ru] & }
\end{equation*}
Thus, in the notation of Definition~\ref{def_effirr} we have 
\begin{equation}
\label{eq_Iirr}
I^{\rm irr}(X)=\{E_4,E_5,E_6,E_7\},
\end{equation}
$$R^{\rm irr}(X)=\{ L_{167};\quad L_{124}, E_1-E_6, L_{135};\quad E_2-E_4, L_{237}, E_3-E_5\}$$
and using Proposition~\ref{prop_roots} we get that 
\begin{multline}
\label{eq_Reff}
R^{\rm eff}(X)=\{ L_{167};\quad L_{124}, E_1-E_6, L_{135}, L_{246},L_{356},2L-E_{123456};\\
 E_2-E_4, L_{237}, E_3-E_5, L_{347},L_{257},L_{457}\}.
\end{multline}

Consider the sequence
$$A=(L_{14}, L_{567}, -L_{467}, 2L-E_{123467}, E_2, L_{245},  -L_{345}, 2L-E_{13456}, E_6-E_7, -2L+E_{11457}).$$
One can check that $A$ is a toric system and 
$$A^2=(-1,-2,-2,-2,-1,-2,-2,-1,-2,-3).$$
Denote $C_i=3L-E_{1234567}-E_i$ and $Q_{ij}=2L-E_{1234567}+E_{ij}$, these are $(-1)$-classes. 
Then by Lemma~\ref{lemma_IFdef}
\begin{multline*}
I(X,A)=\\
=\{A_{1\ldots j}\mid 1\le j\le 4\}\cup \{A_{i\ldots j}\mid 2\le i\le 5, 5\le j\le 7\}\cup
\{A_{i\ldots j}\mid 6\le i\le 8, 8\le j\le 9\}=\\
=\{L_{14}, Q_{23}, L_{15}, C_1\}\cup \{E_2, Q_{25},L_{13},Q_{24}, L_{45}, C_4, Q_{67}, C_5, E_3,Q_{35}, L_{12}, Q_{34}\}\cup \\
\cup \{Q_{27},L_{16}, Q_{37}, Q_{26}, L_{17}, Q_{36}\}.
\end{multline*}
It follows (see \eqref{eq_Iirr}) that there are no $(-1)$-curves in $I(X,A)$  and therefore $A$ is not an augmentation in the weak sense, see Definition~\ref{def_augm2} and the definition of $I(X,A)$ in Section~\ref{section_firstkind}. A forteriori, $A$ is neither a standard augmentation, nor exceptional or strong exceptional augmentation.

On the other hand, $A$ is strong exceptional. To check this, we use the next 
\begin{lemma}
\label{lemma_910}
Let $X$ be a weak del Pezzo surface of degree $2$. Let $A$ be a toric system on~$X$ with 
$$A^2=(-1,-2,-2,-2,-1,-2,-2,-1,-2,-3).$$
Then $A$ is strong exceptional if and only if all $(-2)$-classes $A_{k\ldots l}$ with $1\le k\le l\le 9$ are neither effective nor anti-effective and the $(-3)$-classes $A_{9,10}$ and $A_{10}$ are not anti-effective.
\end{lemma}
\begin{proof}
Direct consequence of Theorem~\ref{theorem_checkonlyminustwo}.
\end{proof}

By the above lemma, we have to check that the $(-2)$-classes
\begin{multline*}
A_2=L_{567}, A_3=-L_{467}, A_4=2L-E_{123467}, A_{23}=E_4-E_5, A_{34}=L_{123}, A_{234}=2L-E_{123567}, \\
A_6=L_{245}, A_7=-L_{345}, A_{67}=E_3-E_2, A_9=E_6-E_7
\end{multline*}
are neither effective nor anti-effective and that the $(-3)$-classes 
$$A_{10}=-(2L-E_{11457}), A_{9,10}=-(2L-E_{11456})$$
are not anti-effective.
The first claim is straightforward by \eqref{eq_Reff}, for the second the argument is based on the following trivial fact.
\begin{lemma}
\label{lemma_bite}
Suppose  $D$ is an effective divisor on a surface and $C$ is an irreducible reduced curve. If $D\cdot C<0$ then $D-C$ is also an effective divisor.
\end{lemma}
Suppose now that $D=-A_{10}=2L-E_{11457}\ge 0$. Since $D\cdot L_{167}=-1$ and $L_{167}\in R^{\rm irr}(X)$ it follows from Lemma~\ref{lemma_bite} that 
$D-L_{167}=L_{145}+E_6\ge 0$. Since $(L_{145}+E_6)\cdot E_6=-1$ and $E_6\in I^{\rm irr}(X)$, it follows from Lemma~\ref{lemma_bite} that $(L_{145}+E_6)-E_6=L_{145}\ge 0$, which contradicts to \eqref{eq_Reff}.
Similarly (with $6$ and $7$ interchanged) one checks that $-A_{9,10}=2L-E_{11456}$ is not effective.
Therefore, $A$ is a strong exceptional toric system by Lemma~\ref{lemma_910}.

\begin{remark} 
Note that the surface $X$ was obtained from $\P^2$ in two steps of blow-ups. It agrees with Conjecture~\ref{conj_twosteps}.
\end{remark}

\appendix
\section{Classification of weak del Pezzo surfaces}
\label{section_app}

In this Appendix we present classification of weak del Pezzo surfaces of degrees $7$ to $3$. Our references here were \cite{CT} and \cite{De}. Recall that two weak del Pezzo surfaces $X$ and $X'$ are \emph{of the same type} if there exists an isometry $\Pic(X)\to \Pic(X')$ preserving canonical class and the sets of negative curves. The set of $(-2)$-curves on $X$ is a set of simple roots of some root subsystem in $R(X)$, see Proposition~\ref{prop_roots}. It turns out (see \cite[Remark 2.1.4]{De}) that for surfaces $X$ of degree $\ge 3$ the type is uniquely determined by $d=\deg(X)$, the incidence graph $\Gamma$ of $(-2)$-curves and the number $m$ of $(-1)$-curves. We denote such type by $(d,\Gamma,m)$ or just by $(\Gamma,m)$ when the degree is clear from the context.  We omit $m$ if $d$ and $\Gamma$ determine the type uniquely.
Any weak del Pezzo surface (except for Hirzebruch surfaces $\mathbb F_0$ and $\mathbb F_2$) is a blow-up of $\P^2$ (maybe at infinitesimal points), we denote the standard basis in $\Pic(X)$ by $L,E_1,\ldots,E_n$. In the following tables, for any type of degrees $7$ to $3$, we present the set $R^{\rm irr}(X)\subset R(X)$ of $(-2)$-classes in $\Pic(X)$ corresponding to irreducible $(-2)$-curves, where $X$ is a certain weak del Pezzo surface of that type. 

By \cite[a table in 25.5.2]{Ma}, $R(X)=\{\pm(E_1-E_2)\}$ for $X$ of degree $7$, $R(X)=\{E_i-E_j, \pm L_{ijk}\}$ for $X$ of degree $6,5,4$, and 
$R(X)=\{E_i-E_j, \pm L_{ijk}, \pm Z\}$ for $X$ of degree $3$,  where $L_{ijk}=L-E_i-E_j-E_k$ and 
$Z:=2L-E_1-E_2-E_3-E_4-E_5-E_6$.

\begin{table}[h]
\begin{center}
\caption{Weak del Pezzo surfaces of degree $7$}
\label{table_d7}
\begin{tabular}{|p{3cm}|p{6cm}|}
 \hline
 type & $R^{\rm irr}$  \\
 \hline
 {$\emptyset$}& $\emptyset$   \\
 \hline
 {$A_1$}& $E_1-E_2$   \\
 \hline
\end{tabular}
\end{center}
\end{table}

\begin{table}[h]
\begin{center}
\caption{Weak del Pezzo surfaces of degree $6$}
\label{table_d6}
\begin{tabular}{|p{3cm}|p{6cm}|}
 \hline
 type & $R^{\rm irr}$ \\
 \hline
 {$\emptyset$}& $\emptyset$  \\
 \hline
 {$A_1,4$}& $E_1-E_2$ \\
 \hline	
 {$A_1,3$ }& $L_{123}$  \\
 \hline
  {$2A_1$ } & $E_1-E_2,L_{123}$  \\
 \hline
   {$A_2$} & $E_1-E_2,E_2-E_3$  \\
 \hline
   {$A_1+A_2$} & $L_{123},E_1-E_2,E_2-E_3$   \\
 \hline
\end{tabular}
\end{center}
\end{table}

\begin{table}[h]
\begin{center}
\caption{Weak del Pezzo surfaces of degree $5$}
\label{table_d5}
\begin{tabular}{|p{3cm}|p{6cm}|}
 \hline
 type & $R^{\rm irr}$   \\
 \hline
 {$\emptyset$}& $\emptyset$  \\
\hline
 {$A_1$ }& $E_1-E_2$ \\
 \hline
  {$2A_1$} & $E_1-E_2,E_3-E_4$\\
 \hline
   {$A_2$} & $E_1-E_2,E_2-E_3$ \\
 \hline
   {$A_1+A_2$} & $L_{123},E_1-E_2,E_2-E_3$ \\
 \hline
   {$A_3$} & $E_1-E_2,E_2-E_3,E_3-E_4$ \\
 \hline
   {$A_4$} & $E_1-E_2,E_2-E_3,E_3-E_4,L_{123}$  \\
 \hline
\end{tabular}
\end{center}
\end{table}

\begin{table}[h]
\begin{center}
\caption{Weak del Pezzo surfaces of degree $4$}
\label{table_d4}
\begin{tabular}{|p{3cm}|p{7cm}|}
 \hline
 type & $R^{\rm irr}$  \\
 \hline
 {$\emptyset$}& $\emptyset$  \\
 \hline
 {$A_1$ }& $E_4-E_5$    \\
 \hline
 {$2A_1,9$ }& $E_2-E_3, E_4-E_5$ \\
  \hline
 {$2A_1,8$} & $L_{123},E_4-E_5$ \\
 \hline
 {$A_2$}& $E_3-E_4,E_4-E_5$ \\
 \hline
 {$3A_1$}& $L_{123},E_2-E_3,E_4-E_5$  \\
 \hline
 {$A_1+A_2$}& $E_1-E_2,E_3-E_4,E_4-E_5$   \\
 \hline
 {$A_3,5$}& $E_2-E_3,E_3-E_4,E_4-E_5$  \\
 \hline
 {$A_3,4$} & $L_{123},E_3-E_4,E_4-E_5$ \\
 \hline
 {$4A_1$ } & $E_1-E_2,E_4-E_5,L_{123},L_{345}$ \\
  \hline 
 {$2A_1+A_2$} & $E_4-E_5,L_{123},E_1-E_2,E_2-E_3$   \\
  \hline
 {$A_1+A_3$} & $E_1-E_2,L_{123},E_3-E_4,E_4-E_5$   \\
  \hline
 {$A_4$} & $E_1-E_2,E_2-E_3,E_3-E_4,E_4-E_5$   \\
  \hline
  {$2A_1+A_3$} & $E_1-E_2,L_{345},L_{123},E_3-E_4,E_4-E_5$  \\
  \hline
  {$D_4$} & $E_2-E_3,E_3-E_4,E_4-E_5,L_{123}$   \\
  \hline
 {$D_5$} & $E_1-E_2,E_2-E_3,E_3-E_4,E_4-E_5,L_{123}$\\
  \hline 
\end{tabular}
\end{center}
\end{table}

\begin{table}[h]
\begin{center}
\caption{Weak del Pezzo surfaces of degree $3$}
\label{table_d3}
\begin{tabular}{|p{3cm}|p{9cm}|}
 \hline
 type & $R^{\rm irr}$  \\
 \hline
 {$\emptyset$}& $\emptyset$   \\
 \hline
 {$A_1$ }& $Z$    \\
 \hline
 {$2A_1$ }& $E_1-E_2, E_3-E_4$  \\
 \hline
 {$A_2$}& $E_1-E_2,E_2-E_3$  \\
 \hline
 {$3A_1$}& $E_1-E_2,E_3-E_4,E_5-E_6$ \\
  \hline
 {$A_1{+}A_2$}& $E_4-E_5,E_1-E_2,E_2-E_3$  \\
  \hline
 {$A_3$}& $E_1-E_2,E_2-E_3,E_3-E_4$  \\
  \hline
 {$4A_1$} & $E_1-E_2,E_3-E_4,E_5-E_6,Z$  \\
  \hline 
 {$2A_1{+}A_2$} & $E_4-E_5,L_{123},E_1-E_2,E_2-E_3$  \\
  \hline
 {$A_1{+}A_3$} & $E_5-E_6,E_1-E_2,E_2-E_3,E_3-E_4$   \\
  \hline
 {$2A_2$} & $E_1-E_2,E_2-E_3,E_4-E_5,E_5-E_6$ \\
  \hline
 {$A_4$} & $E_1-E_2,E_2-E_3,E_3-E_4,E_4-E_5$ \\
 \hline
 {$D_4$} & $E_1-E_2,E_3-E_4,E_5-E_6,L_{135}$ \\
 \hline
 {$2A_1+A_3$} & $E_5-E_6,Z,E_1-E_2,E_2-E_3,E_3-E_4$  \\
 \hline
 {$A_1+2A_2$} & $L_{123},E_1-E_2,E_2-E_3,E_4-E_5,E_5-E_6$ \\
 \hline
 {$A_1+A_4$} & $Z,E_1-E_2,E_2-E_3,E_3-E_4,E_4-E_5$  \\
 \hline
 {$A_5$} & $E_1-E_2,E_2-E_3,E_3-E_4,E_4-E_5,E_5-E_6$ \\
 \hline
 {$D_5$} & $E_1-E_2,E_2-E_3,E_3-E_4,E_4-E_5,L_{126}$  \\
 \hline 
 {$3A_2$} & $E_1-E_2,E_2-E_3,E_4-E_5,E_5-E_6,L_{123},L_{456}$  \\
 \hline
 {$A_1+A_5$} & $Z,E_1-E_2,E_2-E_3,E_3-E_4,E_4-E_5,E_5-E_6$  \\
 \hline
 {$E_6$} & $L_{123},E_1-E_2,E_2-E_3,E_3-E_4,E_4-E_5,E_5-E_6$  \\
 \hline
\end{tabular}
\end{center}
\end{table}

\end{document}